\documentclass[12pt,english]{amsart}
\usepackage{amsmath}
\usepackage{ulem} 
\calclayout
\usepackage{float}
\usepackage{amsfonts,amssymb}
\usepackage{graphicx, epsfig, psfrag}
\usepackage{epstopdf}
\usepackage{color}
\usepackage{enumerate}
\usepackage{tikz}
\usepackage{tikz, pgfplots}
\usepackage{hyperref}
\usepackage{cleveref}
\theoremstyle{plain}
\newtheorem{theorem}{Theorem}
\newtheorem{lemma}[theorem]{Lemma}

\newtheorem{remark}[theorem]{Remark}

\newcommand\A{\mathcal{A}}
\newcommand\T{\mathcal T}
\newcommand\RR{\mathbb R}
\newcommand\NN{\mathbb N}

\newcommand\PP{\mathbb P}

\newcommand\EE{\mathcal E}

\newcommand\nn{{\boldsymbol n}}

\newcommand\xx{{\boldsymbol x}}

\newcommand\I{\mathcal{I}}

\newcommand{\trih}{\mathcal{T}_h}

\newcommand{\triho}{\mathcal{T}_{h_0}} 

\newcommand{\triH}{{\mathcal{T}_H}}

\newcommand{\btriH}{{\partial\mathcal{T}_H}}

\newcommand\dO{\partial\Omega}
\newcommand\dtau{{\partial\tau}}
\newcommand\OO{\Omega}

\newcommand\amin{a_{\text{min}}}
\newcommand\amax{a_{\text{max}}}

\newcommand\MsFEM{\text{MsFEM}}

\newcommand\tH{{\widetilde H}}

\newcommand\Hdiv{H({\rm{div},\OO)}}

\newcommand\Hdivtau{{H({\rm{div},\tau)}}}

\newcommand\VVRM{\PP_0(\triH)}

\newcommand\tL{{\widetilde\Lambda}}
 
\newcommand\VH{{H^1(\triH)}}
\newcommand\tVh{{\widetilde V_h}}
\newcommand\tVhz{{\widetilde V_{h_0}}}
\newcommand\tVV{{\widetilde H^1(\triH)}}

\def\div{\operatorname{div}}

\def\E{\operatorname{\mathcal E}}

\def\bgrad{\operatorname{\boldsymbol{\operatorname{\nabla}}}} 
\def\span{\operatorname{span}}

\newcommand\bsigma{{\boldsymbol\sigma}}
\newcommand\uurm{u^0}
\newcommand\vvrm{v^0}
\newcommand\teta{{\widetilde\eta}}
\newcommand\lambdarm{\lambda^0}
\newcommand\tlambda{\widetilde\lambda}

\newcommand\murm{{\mu^0}}

\newcommand\tmu{{\widetilde\mu}}

\newcommand\tphi{{\widetilde\phi}}

\newcommand\txi{{\widetilde\xi}}
\newcommand\tv{{\widetilde v}}
\newcommand\vv{{\boldsymbol v}}

\definecolor{bluegreen}{rgb}{0,0.75,0.75}

\begin{document}
\title{A three-field Multiscale Method}
\author{Franklin de Barros, Alexandre L. Madureira, Fr\'ed\'eric Valentin}
\address{Laborat\'orio Nacional de Computa\c c\~ao Cient\'\i fica - LNCC\\Petr\'opolis, Brazil}
\email{fcbarros@lncc.br,alm@lncc.br,valentin@lncc.br}
\begin{abstract}\color{black}Inspired by the seminal paper \emph{A Three-Field Domain Decomposition Method} by F. Brezzi and L. D. Marini, we propose the Multiscale-Hybrid-Hybrid Method (MH$^2$M) for the Darcy problem. After a sequence of formal manipulations, the resulting multiscale finite element method leads to a symmetric positive definite formulation posed solely in terms of the trace variable. We establish stability and convergence results for a family of finite element spaces and highlight connections with other multiscale finite element methods. \color{black} 
\end{abstract}
\maketitle
%
%
\section{Introduction}

Simulation of fluid flow in heterogeneous domains, as found in oil reservoirs, transport of contaminants and water resources issues, is a big challenge since fine meshes are needed to capture the multiscale solution. The Darcy equation is the most representative model in the field of porous media, which in its primary formulation is composed of a Poisson equation with a multiscale coefficient. Precisely, the Poisson problem consists of finding the weak solution $u:\OO\to\RR$ of 
\begin{equation}\label{e:elliptic}
\begin{gathered}
-\div\A\bgrad u=f\quad\text{in }\OO,
\\
u=0\quad\text{on }\partial\OO,
\end{gathered}
\end{equation}
where $\OO\subset\RR^d$ for $d=2$ or $3$ is an open bounded domain with a polyhedral Lipschitz boundary $\partial\Omega$, and $f\in L^2(\OO)$. The symmetric tensor $\A\in[L^\infty(\OO)]_{\text{sym}}^{d\times d}$ is uniformly positive definite and let the positive constants $\amin$ and $\amax$ be such that 
\begin{equation}\label{e:bounds}
\amin|\vv|^2\le\A(\xx)\,\vv\cdot\vv\le\amax|\vv|^2\quad\text{for all }\vv\in\RR^d
\end{equation}
for almost all $\xx\in\OO$. The tensor $\A$ may include multiple scales.

The flow velocity, represented by the pressure gradient, is usually the quantity of interest. To approximate it, the mixed version of Darcy's equation is preferable to its elliptic~\eqref{e:elliptic} version, for which the classical inf-sup stable pairs of finite elements are the Raviart-Thomas element (RT$_k$)~\cite{RavTho77}
or the element Brezzi-Douglas-Marini (BDM$_k$) \cite{BreDouMar85,BofBreFor13}, which prescribes the continuity of the
normal velocity component combined with discontinuous pressure interpolations. In addition to being stable, these finite elements preserve local mass, a fundamental property in practical engineering applications. However, since such choices produce linear systems that represent
saddle point problems, efficient solver are harder to come by~\cite{BreForMar06}.
A way to circumvent this problem is by using  hybridization~\cite{BofBreFor13} techniques, which start with a discontinuous version of a stable element (RT$_k$ or BDM$_k$), forcing the weak continuity of the normal components through Lagrange multipliers. Velocity and pressure are now fully discontinuous and can be eliminated at the element level, leading to a positive definite symmetrical system where the only unknowns are the Lagrange multipliers~\cite{ArnBre85}. 

The publication of the seminal work by \color{black} Brezzi \color{black} and Marini~\cite{MR1262602} was followed by a series of papers investigating, in particular, preconditioning and stabilization techniques that allow for different combinations of space~\cite{MR1262602,MR1228218,AgoTho95,Rapin2004,MR3475654,10.5555/868782,MR1826573,MR1888832,MR1795297,MR1961881,MR2031400}. 
Although these concepts are related, they are not the same, specially when the estimates depend on the number of subdomains~\cite{MR1961881,MR2031400,MR3475654}. Note that stabilization is crucial in those papers, an ingredient that we \color{black} avoid. \color{black} Indeed, in principle, stabilization allows for more flexible interpolation, but \color{black} often \color{black} imposes restrictions on the meshes and interpolation spaces, e.g.~\cite{MR1228218}, or involves the inclusion of terms and stabilization parameters of nontrivial computation~\cite{MR1795297}.

Hybridization has also been used to construct and analyze multiscale methods, which have become an attractive option for dealing with primal or mixed weak forms of \eqref{e:elliptic}  when the standard Galerkin method based on continuous polynomial interpolation becomes computationally expensive. The multiscale methods share the commonality of having global problems built on the solution of local problems that scale up submesh structures. An attractive feature is that the global problem dimension is multiscale independent, and the basis functions are independent of each other and can be computed in parallel. An hybridization of the primal form of~\eqref{e:elliptic}, following~\cite{RavTho77b} yields the \emph{Multiscale Hybrid-Mixed} (MHM for short) family of methods~\cite{HarParVal13, AraHarParVal13a}, where the local problems are of Neumann type. On the contrary, the hybridization of the dual version of \eqref{e:elliptic} leads to the multiscale mortar method~\cite{ArbPenWheYot07} and the multiscale version of the HDG method in~\cite{EfeLazShi15}. Local problems are of the Dirichlet type in these cases. Other alternative multiscale methods in the context of the Darcy model (or Poisson equation) are the VMS method~\cite{HugFeiMazQui98a}, the MsFEM and GMsFEM~\cite{EfeHouWu00}, the PGEM and GEM~\cite{FerHarParVal12}, the HMM~\cite{EEng03}, the LOD and LSD method~\cite{MalPet14,MadSar21} and other variants of the MHM method like~\cite{DurDevGomVal18a}, to name a few.

In this work, we follow the strategy used to build the MHM method, which starts at the continuous level placed on a coarse partition. It consists of decomposing the exact solution into local and global contributions that, when discretized, dissociate local and global problems: the global formulation is defined on the skeleton of the coarse partition, yielding the degrees of freedom; the local problems provide the multiscale basis functions and their computation is embarrassingly parallel as the local problems are independent of each other. However, the MHM method for~\eqref{e:elliptic} has a saddle point structure similar to Galerkin's method for the mixed form of \eqref{e:elliptic}. Notably, its lower order version generalizes the Galerkin's method with the element RT$_0$ in simplexes for the case of discretized multiscale coefficient problems with polytopal meshes~\cite{BarJaiParVal20}.

This work revisits the three field method of F.~Brezzi and D.~Marini~\cite{MR1262602} in the context of the MHM method. Specifically, we relax the continuity of the flow variable $\lambda$ and weakly enforce it through the action of a second Lagrange multiplier $\rho$. Then, we decompose the broken infinite dimensional spaces where $u$ and $\lambda$ belong in such a way that the discretization results in a new multiscale finite element method that

\begin{itemize}
\item  induces a definite positive symmetric global linear system to compute the degrees of freedom. It differs from the original MHM method, which induces a linear saddle point system;
\item approximates the flux variable $\lambda$ using a multiscale flux basis obtained from a Dirichlet-to-Neumann operator defined from new local Neumann problems driven by the basis functions of $\rho$. This feature is new when compared to the MHM method strategy where the basis for $\lambda$ is chosen in an ad-hoc fashion and ``ignores'' underlying physical properties; 
\item imposes weak continuity of discrete primal (pressure) and dual (flow) variables on the skeleton of the coarse partition, with a discrete flow that is in local equilibrium with external forces. 
\end{itemize}

In addition to the construction of the new multiscale method, called the Multiscale Hybrid-Hybrid Method (MH$^2$M for short), the main contributions of this work are:

\begin{itemize}
\item[$(i)$] demonstrate the well-posedness and the best approximation results for the MH$^2$M under abstract compatibility conditions between interpolation spaces;
\item[$(ii)$] propose families of interpolation spaces that fulfill the conditions in $(i)$ and for which the optimal convergence of MH$^2$M is demonstrated showing the influence of the different mesh parameters;
\item[$(iii)$] bridge the MH$^2$M with other multiscale finite elements, notably the MHM method and MsFEM~\cite{MR2477579,EfeHouWu00};
\item[$(iv)$] numerical validation and comparison between MsFEM, MHM, and the proposed method, with particular emphasis on the number of global degrees of freedom required to achieve a given error level.
\end{itemize}
This work is organized as follows:
Section~\ref{sec:setting} introduces the functional framework associated with domain partitioning. In Section~\ref{sec:exact}, we characterize the exact solution as the combination of local and global problems.
Section~\ref{sec:discrete} presents the MH$^2$M method and discusses its connection to MsFEM. The mathematical formulation used to establish well-posedness and best-approximation properties is developed in Section~\ref{sec:analysis}.
In Section~\ref{sec:compatible}, we introduce families of interpolation spaces that meet the requirements set in Section~\ref{sec:analysis}, and we prove optimal convergence rates under local regularity assumptions. The algorithmic implementation of the method is detailed in Section~\ref{sec:basis}. Numerical validations and comparisons are presented in Section~\ref{sec:conv}. Finally, conclusions are summarized in Section~\ref{sec:concl}, and additional technical results are provided in the Appendix.
%
%
\section{Settings}
\label{sec:setting}

\subsection{Partition, broken spaces and norms}
\label{sec:partition}

Let  $\triH$ be a regular mesh, which can be based on different element geometries, and $\btriH$ be the skeleton of $\triH$. Without loss of generality, we adopt above and in the remainder of the text, the terminology of three-dimensional domains, denoting for instance the boundaries of the elements by faces. For a given element $\tau\in\triH$, with diameter $H_\tau$, let $\partial\tau$ denote its boundary and $\nn^\tau$ the unit size normal vector that points outward $\tau$. \color{black} We define the mesh size by $H=\max_{\tau\in\triH} H_\tau$, and \color{black}  denote by $\nn$ the outward normal vector on $\dO$. Consider now the following spaces:
\begin{equation}\label{e:spaces}
\begin{gathered}
  \VH:=\{v\in L^2(\OO):\,v|_\tau\in H^1(\tau),\,\tau\in\triH\},
  \qquad
  \Lambda:=\prod_{\tau\in\triH}H^{-1/2}(\partial\tau),
  \\
  H_0^{1/2}(\btriH):=\{v|_{\btriH}:\,v\in H_0^1(\OO)\}. 
\end{gathered}
\end{equation}
For $w$, $v\in L^2(\OO)$ and $\rho\in H_0^{1/2}(\btriH)$, $\mu\in\Lambda$ define 
\begin{gather*}
  (w,v)_{\triH}:=\sum_{\tau\in\triH}\int_\tau wv\,d\xx
  \quad\text{and}\quad
  \langle\mu,\rho\rangle_{\btriH}:=\sum_{\tau\in\triH}\langle\mu_\tau,\rho\rangle_{\partial \tau},
\end{gather*}
where $\langle\cdot,\cdot\rangle_{\partial \tau}$ is the dual product involving $H^{-1/2}(\dtau)$ and $H^{1/2}(\dtau)$, defined by
\[
\langle\mu,\rho\rangle_{\partial \tau}:=\int_\tau\div\bsigma v\,d\xx+\int_\tau\bsigma\cdot\bgrad v \,d\xx
\]
for all $\bsigma\in\Hdivtau$ such that $\bsigma\cdot\nn^\tau=\mu$, and all $v\in H^1(\tau)$ such that $v|_\dtau=\rho$. Note that the null space of the operator $\langle\mu,\cdot\rangle_{\btriH}$ is $H^{1/2}_0(\btriH)$, if $\mu$ is skeleton trace of the normal component of some function in $\Hdiv$. We use the same notation for a function in $H_0^{1/2}(\btriH)$ and its restriction to an element boundary $\dtau$. 

Consider also the following (semi-) norms: 
\begin{equation}\label{e:normsdef}
\begin{gathered}
|v|_{H_\A^1(\tau)}:=\|\A^{1/2}\bgrad v\|_{L^2(\tau)}, 
\qquad
|v|_{H_\A^1(\Omega)}^2:=\sum_{\tau\in\triH}|v|_{H_\A^1(\tau)}^2, 
\\
|\xi|_{H^{1/2}(\dtau)}
:=\inf_{\substack{\phi\in H^1(\tau)\\\phi|_\dtau=\xi}}|\phi|_{H_\A^1(\tau)},
\qquad
\|\rho\|_{H^{1/2}(\btriH)}
:=\inf_{\substack{\phi\in H_0^1(\OO)\\\phi|_\btriH=\rho}}|\phi|_{H_\A^1(\OO)},
\end{gathered}
\end{equation}
where $v\in H^1(\triH)$, $\xi\in H^{1/2}(\dtau)$ and $\rho\in H_0^{1/2}(\btriH)$. Also, for $\mu\in\Lambda$ let
\begin{equation}
\label{norm12d}
|\mu|_{H^{-1/2}(\dtau)}
:=\sup_{\tphi\in\tH^{1/2}(\dtau)}\frac{\langle\mu,\tphi\rangle_\dtau}{|\tphi|_{H^{1/2}(\dtau)}},
\qquad
|\mu|_\Lambda^2:=\sum_{\tau\in\triH}|\mu|_{H^{-1/2}(\dtau)}^2,
\end{equation}
where 
\begin{equation}\label{e:tHdef}
\tH^{1/2}(\dtau):=\biggl\{\xi\in H^{1/2}(\dtau):\,\int_\dtau\xi\,d\xx=0\biggr\}. 
\end{equation}

%
%
\subsection{Hybridization}
\label{hybrid}

Let $u\in\VH$, $\rho\in H_0^{1/2}(\btriH)$, and $\lambda\in\Lambda$ be such that 
\begin{equation}\label{e:weak-hybridH}
\begin{alignedat}{5}
&\quad(\A\bgrad u,\bgrad v)_{\triH}&&-\langle\lambda,v\rangle_{\btriH}&&&&=(f,v)_{\triH}\quad&&\text{for all }v\in\VH,
  \\
&-\langle\mu,u\rangle_{\btriH}&&&&+\langle\mu,\rho\rangle_{\btriH}&&=0\quad&&\text{for all }\mu\in\Lambda,
  \\
&&&\quad\langle\lambda,\xi\rangle_{\btriH}&&&&=0\quad&&\text{for all }\xi\in H_0^{1/2}(\btriH).
\end{alignedat}
\end{equation}
This is the formulation proposed in~\cite{MR1262602,MR1826573}, and yields a hybrid formulation for~\eqref{e:elliptic}. \color{black} Indeed, $u$, the weak solution of~\eqref{e:elliptic}, $\rho=u|_{\btriH}$ and $\lambda=\A\bgrad u\cdot\nn^\tau|_{\btriH}$ solve~\eqref{e:weak-hybridH} \color{black}. Classical results on Sobolev spaces~\cite{BofBreFor13,gatica14} yield the converse statement. If the triplet $(u,\rho,\lambda)$ solves~\eqref{e:weak-hybridH}, then, from the third equation, $\lambda=\bsigma\cdot\nn^\tau$ for some $\bsigma\in\Hdiv$. We gather from the first equation that $-\div\A\bgrad u=f$ weakly in each element, and that $\lambda=\A\bgrad u\cdot\nn^\tau|_{\btriH}$. Finally, it follows from the second equation that $\rho=u|_{\btriH}$. Then, $u\in H_0^1(\Omega)$ and~\eqref{e:elliptic} holds in the weak sense. It follows from the above arguments that existence and uniqueness for solutions of~\eqref{e:weak-hybridH} is immediate.

\section{Exact solution decomposition}
\label{sec:exact}
Consider the decomposition 
\begin{equation*}\label{e:VHdecomp}
\VH=\VVRM\oplus\tVV,
\end{equation*}
where $\VVRM$ is the space of piecewise constants, and $\tVV$ is the space of functions with zero average within each element border $\dtau$, for $\tau\in\triH$. Note that this decomposition differs from that of~\cite{AraHarParVal13a,FETI,MR2282408,HarMadVal16,HarParVal13}.

Taking a further step, we decompose  $\Lambda$ into a space of ``constants'' plus ``zero-average'' functionals over the border of the elements of $\triH$. For each $\tau_i\in\triH$, let $\lambdarm_i\in\Lambda$ be such that 
\begin{equation}\label{e:lrmdef}
\langle\lambdarm_i,v\rangle_{\btriH}:=\int_{\partial \tau_i}v\,d\xx\quad\text{for all }v\in H^1(\triH). 
\end{equation}
Let $N$ be the number of elements of $\triH$ and 
\begin{gather*}\label{e:lammbdadecomp}
\Lambda^0 :=\span\{\lambdarm_i:\,i=1,\dots,N\},
\qquad 
\tH^{-1/2}(\dtau):=\{\tmu\in H^{-1/2}(\dtau):\,\langle\tmu,1\rangle_{\dtau}=0\},   
\\
\tL:=\prod_{\tau\in\triH}\tH^{-1/2}(\dtau)=\{\tmu\in\Lambda:\,\langle\tmu,\vvrm\rangle_{\btriH}=0\text{ for all }\vvrm\in\VVRM\}. 
\end{gather*}

We can now decompose $\Lambda=\Lambda^0\oplus\tL$ as follows~\cite{MR2759829}. Given $\mu\in\Lambda$, let $\murm\in\Lambda^0$
\begin{equation*}
\langle\murm,\vvrm\rangle_{\btriH}=\langle\mu,\vvrm\rangle_{\btriH}, 
\end{equation*}
for all $\vvrm\in\VVRM$. Observe that $\murm$ is well-defined and $\murm|_{\dtau} = \frac{1}{|\dtau|} \langle\mu,1\rangle_{\dtau}$ for all $\tau\in \triH$. Next, we define uniquely $\tmu := \mu-\murm$, i.e., 
\begin{equation*}
\langle\tmu,v\rangle_{\btriH}=\langle\mu,v\rangle_{\btriH}-\langle\murm,v\rangle_{\btriH}, 
\end{equation*}
and note that $\tmu\in\tL$ since $\langle\tmu,\vvrm\rangle_{\btriH}=\langle\mu,\vvrm\rangle_{\btriH}-\langle\murm,\vvrm\rangle_{\btriH}=0$. 

We then write $u=\uurm+\widetilde u$, where $\uurm\in\VVRM$ and $\widetilde u\in\tVV$, and also $\lambda=\lambdarm+\tlambda$, for $\lambdarm\in\Lambda^0$ and $\tlambda\in\tL$. Owing to such decomposition, we can characterize each component of $u$ and $\lambda$ from data and Lagrange multipliers using \eqref{e:weak-hybridH}. To see that, we first test~\eqref{e:weak-hybridH} with $(v^0,\mu^0,0)$, and gather  that $\lambdarm\in\Lambda^0$ and $u^0\in\VVRM$ solve 
\begin{equation}\label{e:weak-e1.easy}
\begin{aligned}
\langle\lambdarm,v^0\rangle_{\btriH}&=-\langle f,v^0\rangle_{\triH}\quad\text{for all }v^0\in\VVRM,
  \\
\langle\mu^0,u^0\rangle_{\btriH}&=\langle\mu^0,\rho)_{\btriH}\quad\text{for all }\mu^0\in\Lambda^0.
\end{aligned}
\end{equation}
Then, the first equation of~\eqref{e:weak-e1.easy} defines $\lambdarm$. The piecewise constant $u^0$ is obtained from the second equation of~\eqref{e:weak-e1.easy}, after the computation of $\rho$.  Both can be computed locally as follow
\begin{gather}
\label{u0l0}
\lambdarm|_\dtau = -\frac{1}{|\dtau|} \int_{\tau}  f \:d\xx\quad\text{and}\quad u^0|_\tau = \frac{1}{|\dtau|} \int_{\dtau}  \rho \:d\xx\quad\text{for all }\tau \in\triH.
\end{gather}

Next, testing~\eqref{e:weak-hybridH} with  $(\widetilde v,\tmu,\xi)$,  we gather that $\widetilde u\in\tVV$, $\tlambda\in\tL$ and $\rho\in H_0^{1/2}(\btriH)$ solve
\begin{equation}\label{e:weak-e1}
\begin{alignedat}{5}
&(\A\bgrad\widetilde u,\bgrad\widetilde v)_{\triH}&&-\langle\tlambda,\widetilde v\rangle_{\btriH}&&&&
=(f,\widetilde v)_{\btriH}
\quad&&\text{for all }\widetilde v\in\tVV,
\\
&-\langle\tmu,\widetilde u\rangle_{\btriH}&&&&+\langle\tmu,\rho\rangle_{\btriH}&&
=0
\quad&&\text{for all }\tmu\in\tL, 
\\
&&&\langle\tlambda,\xi\rangle_{\btriH}&&&&
=-\langle\lambda^0,\xi\rangle_{\btriH}
\quad&&\text{for all }\xi\in H_0^{1/2}(\btriH). 
\end{alignedat}
\end{equation}
The first equation of~\eqref{e:weak-e1} allows the introduction of local mappings. Let $T:\tL\to\tVV$ and $\widetilde T:L^2(\Omega)\to\tVV$ be such that, given $\mu\in\tL$, $q\in L^2(\Omega)$ and $\tau\in\triH$, 
\begin{equation}\label{e:definitionT}
\int_\tau\A\bgrad(T\mu)\cdot\bgrad\widetilde v\,d\xx=\langle\mu,\widetilde v\rangle_{\partial \tau}
\quad\text{and}\quad
\int_\tau\A\bgrad(\widetilde Tq)\cdot\bgrad\widetilde v\,d\xx=\int_{\tau}q\widetilde v\:d\xx,
\end{equation}
for all $\widetilde v\in\tVV$. The mappings $T$ and $\widetilde T$ are locally well defined as the bilinear forms in \eqref{e:definitionT} are coercive on $\widetilde H^1(\tau)$. They are locally bounded as follows
\begin{gather}\label{boundT}
|T\tmu|_{H_\A^1(\tau)}=|\tmu|_{H^{-1/2}(\dtau)}\quad\text{and}\quad |\widetilde T q|_{H_\A^1(\tau)}\leq \frac{C_{\mathcal P}}{a_{min}^{1/2}} H_\tau\|q\|_{L^2(\tau)},
\end{gather}
for all $\tmu\in\tL$ and $q \in L^2(\Omega)$, where $C_{\mathcal P}$ is the positive local Poincar\'e  constant independent of $H_\tau$.
The boundedness of $T$ in \eqref{boundT} is a consequence of  Lemma~\ref{l:ids} (see the Appendix). The second estimate in \eqref{boundT} follows from the definition of $\widetilde T$ in \eqref{e:definitionT}, the Cauchy-Schwartz and local Poincar\'e inequalities.

 Then $\widetilde u=T\tlambda+\tilde Tf$, and the remaining equations in~\eqref{e:weak-e1} yield 
\begin{equation}\label{e:weak-e2}
\begin{alignedat}{3}
  -&\langle\tmu,T\tlambda\rangle_{\btriH}+\langle\tmu,\rho\rangle_{\btriH}&&
  =\langle\tmu,\widetilde Tf\rangle_{\btriH}\quad&&\text{for all }\tmu\in\tL
  \\
  &\langle\tlambda,\xi\rangle_{\btriH}&&
  =-\langle\lambda^0,\xi\rangle_{\btriH}\quad&&\text{for all }\xi\in H_0^{1/2}(\btriH).
\end{alignedat}
\end{equation}
It follows from~\eqref{boundT} that the bilinear form $\langle\cdot,T\cdot\rangle_{\btriH}$ is coercive on $\tL$ locally. Then, the first equation \eqref{e:weak-e2} also induces a local mapping. We denote it  by
\[
G:\{v\in L^2(\btriH):\,v\in H^{1/2}(\dtau),\,\tau\in\triH\}\to\tL
\]
and is such that,  given $\phi\in H^{1/2}(\dtau)$, we set $\tlambda_\phi:=G\phi\in\tH^{-1/2}(\dtau)$ as the (unique) solution of 
\begin{equation}\label{e:Gdef}
\int_\tau\A\bgrad(T\tlambda_\phi)\cdot\bgrad T\tmu\,d\xx
=\langle\tmu,T\tlambda_\phi\rangle_\dtau
=\langle\tmu,\phi\rangle_\dtau\quad\text{for all }\tmu\in\tL. 
\end{equation}
The operator $G$ is bounded locally as follow (see Lemma~\ref{l:ids} in the Appendix): 
\begin{equation}
\label{boundG}
|G\xi|_{H^{-1/2}(\dtau)}=|\xi|_{H^{1/2}(\dtau)}\quad\text{for all }\xi\in H^{1/2}(\dtau).
\end{equation}


Next, the first equation in~\eqref{e:weak-e2} implies that $\tlambda=G(\rho-\widetilde Tf)$, and then, the last equation of~\eqref{e:weak-e2} reads
\begin{equation}\label{e:rhodef}
\langle G\rho,\xi\rangle_\btriH
=-\langle\lambda^0,\xi\rangle_\btriH+\langle G\widetilde Tf,\xi\rangle_{\btriH}
\quad\text{for all }\xi\in H_0^{1/2}(\btriH).
\end{equation}

Define the bilinear forms $g_\tau:H^{1/2}(\dtau)\times H^{1/2}(\dtau)\to\RR$ for $\tau\in\triH$, and $g:H_0^{1/2}(\btriH)\times H_0^{1/2}(\btriH)\to\RR$ by
\begin{equation}\label{e:gdef}
g_\tau(\xi,\phi):=\langle G\xi,\phi\rangle_\dtau,
\qquad
g(\xi,\phi):=\sum_{\tau\in\triH}g_\tau(\xi,\phi)
\quad\text{for $\xi$, $\phi\in H_0^{1/2}(\btriH)$}. 
\end{equation}

\begin{remark}[$TG$ local mapping] 
Note that both $T$ and $G$ are local operators.  
Furthermore, we find $\lambda^0$ and $u^0$ locally and trivially from \eqref{u0l0}, as there are a finite number of unknowns. Thus, ~\eqref{e:rhodef} is the only global, infinite-dimensional equation depending on $\A$ through the operator $G$. For future reference it follows from ~\eqref{e:Gdef} that
\begin{equation}\label{e:tHident}
(TG\xi)|_\dtau=\xi+c_\xi, \quad\text{where }c_\xi=-\frac1{|\dtau|}\int_\dtau\xi\,d\xx
\end{equation}
for all $\xi\in H^{1/2}(\dtau)$ and $\tau\in\triH$.
\end{remark}

\begin{remark}[$TG$ equivalent formulation]
\label{rem:tge}
Note that $G$ is symmetric since, from~\eqref{e:Gdef}, 
\begin{equation}\label{e:Gsymm}
\sum_{\tau\in\triH}\int_\tau\A\bgrad TG\rho\cdot\bgrad TG\xi
=\langle G\rho,TG\xi\rangle_\btriH
=\langle G\rho,\xi\rangle_\btriH,
\end{equation}
for all $\rho,\xi\in H_0^{1/2}(\btriH)$. Also, from~\eqref{e:Gdef} and~\eqref{e:definitionT}, 
\[
\langle G\widetilde Tf,\xi\rangle_{\btriH}
=\langle G\xi,\widetilde Tf\rangle_{\btriH}
=\sum_{\tau\in\triH}\int_\tau\A\bgrad\widetilde Tf\cdot\bgrad TG\xi
=\sum_{\tau\in\triH}\int_\tau fTG\xi. 
\]
Thus, using~\eqref{e:tHident} and~\eqref{u0l0} it holds that 
\begin{multline*}
-\langle\lambda^0,\xi\rangle_\btriH+\langle G\widetilde Tf,\xi\rangle_{\btriH}
=-\sum_{\tau\in\triH}\int_\dtau\lambda^0\xi+\sum_{\tau\in\triH}\int_\tau fTG\xi
\\
=\sum_{\tau\in\triH}\int_\dtau\lambda^0c_\xi+\sum_{\tau\in\triH}\int_\tau fTG\xi
=-\sum_{\tau\in\triH}\int_\tau fc_\xi+\sum_{\tau\in\triH}\int_\tau fTG\xi
=\sum_{\tau\in\triH}\int_\tau f(TG\xi-c_\xi),
\end{multline*}
and formulation \eqref{e:rhodef} is equivalent to 
\begin{gather}\label{e:rhodefa}
 \sum_{\tau\in\triH}\int_\tau \A\nabla TG\rho\cdot\nabla TG\xi
=\sum_{\tau\in\triH}\int_\tau f(TG\xi-c_\xi)
\quad\text{for all }\xi\in H_0^{1/2}(\btriH).
\end{gather}
\end{remark}

In the next theorem, we show that $g_\tau(\cdot,\cdot)$ is coercive and continuous, so \eqref{e:rhodef} is well-posed. Furthermore, we collect the results of the constructive approach described above and provide an explicit characterization of $u$ and $\lambda$ as the solutions to local problems that are brought together via $\rho$, the solution of the global skeletal problem \eqref{e:rhodef}.
  
  \begin{theorem}\label{p:coerc}
Let $g_\tau$, $g$ be defined by~\eqref{e:gdef}. The following coercivity and continuity results hold: 
\begin{equation*}
\begin{aligned}
g_\tau(\txi,\txi)
&=|\txi|_{H^{1/2}(\dtau)}^2\quad\text{for all }\txi\in\tH^{1/2}(\dtau),
\\
g(\xi,\xi)&=|\xi|_{H^{1/2}(\btriH)}^2\quad\text{for all }\xi\in H_0^{1/2}(\btriH),
\\
g(\xi,\rho)&\le|\xi|_{H^{1/2}(\btriH)}|\rho|_{H^{1/2}(\btriH)} \quad\text{for all }\xi,\, \rho\in H_0^{1/2}(\btriH).
\end{aligned}
\end{equation*}
Hence, there exists a unique $\rho\in H_0^{1/2}(\btriH)$ such that
\begin{equation}\label{e:mhhm}
g(\rho,\xi)=-\langle\lambda^0,\xi\rangle_\btriH+g(\widetilde Tf,\xi)
\quad\text{for all }\xi\in H_0^{1/2}(\btriH)
\end{equation}
and 
\[ 
|\rho|_{H^{1/2}(\btriH)} \leq |\lambda^0|_{\Lambda} + |\widetilde T f |_{H^{1/2}(\btriH)}.
\]
Moreover, the exact solution $u$ and $\lambda$ of the hybrid formulation~\eqref{e:weak-hybridH} writes
\begin{gather}
\label{ul-dec}
u= u^0+TG\rho+(I-TG)\widetilde Tf\quad \text{and}\quad \lambda=\lambdarm+G\rho-G\widetilde Tf,
\end{gather}
where $u^0$ and $\lambda^0$ are the piecewise constant functions defined through \eqref{u0l0}.
\end{theorem}
\begin{proof}
For $\txi\in\tH^{1/2}(\dtau)$ let $\tlambda_\txi :=G\txi$. The local coercivity holds since, from~\eqref{e:tHident} and Lemma~\ref{l:ids}-(i) (see the Appendix),
\[
|\txi|_{H^{1/2}(\dtau)}^2
=|T\tlambda_\txi|_{H_\A^1(\tau)}^2
=\langle\tlambda_\txi,\txi\rangle_\dtau
=g_\tau(\txi,\txi). 
\]
Next, assume that $\xi\in H_0^{1/2}(\btriH)$. Then $\xi=\txi+\xi^0$, where $\txi|_\dtau\in\tH^{1/2}(\dtau)$ and $\xi^0|_\dtau$ is constant for each $\tau\in\triH$. Then, from  the definition of mapping $G$ in~\eqref{e:Gdef}, and from the fact $|\cdot|_{H^{1/2}(\dtau)}$ is a semi-norm, we get
\[
g(\xi,\xi)
=\sum_{\tau\in\triH}g_\tau(\xi,\xi)
=\sum_{\tau\in\triH}g_\tau(\txi,\txi)
=\sum_{\tau\in\triH}|\txi|_{H^{1/2}(\dtau)}^2
=\sum_{\tau\in\triH}|\xi|_{H^{1/2}(\dtau)}^2
=|\phi_\xi|_{H_\A^1(\triH)}^2
\]
where we used Lemma~\ref{l:halfnorm} (see the Appendix) in the last step.

To show continuity, let $\xi$, $\rho\in H_0^{1/2}(\btriH)$ and set $\tlambda_\phi:=G\xi$, $\tlambda_\rho:=G\rho$. Then
\begin{multline*}
g(\xi,\rho)
=\sum_{\tau\in\triH}\int_\tau\A\bgrad(T\tlambda_\xi)\cdot\bgrad T\tlambda_\rho\,d\xx
\le\sum_{\tau\in\triH}|T\tlambda_\xi|_{H_\A^1(\tau)}|T\tlambda_\rho|_{H_\A^1(\tau)}
\\
=\sum_{\tau\in\triH}|T\tlambda_\xi|_{H^{1/2}(\dtau)}|T\tlambda_\rho|_{H^{1/2}(\dtau)}
=\sum_{\tau\in\triH}|\xi|_{H^{1/2}(\dtau)}|\rho|_{H^{1/2}(\dtau)}, 
\end{multline*}
where we used Lemma~\ref{l:halfnorm} (see the Appendix), identity~\eqref{e:tHident} and that $|\cdot|_{H^{1/2}(\dtau)}$ is a semi-norm. Existence and uniqueness of solution for~\eqref{e:mhhm} follows from Lax-Milgram's Lemma. The stability result for $\rho$  follows from the coercivity and continuity of $g(\cdot,\cdot)$ and the definition of the  $\Lambda$-norm in \eqref{norm12d}. Finally, the characterization~\eqref{ul-dec} is a straightforward consequence of the decomposition $u = u^0 + \widetilde u$ and $\lambda = \lambda^0 +\tlambda$, using 
\begin{gather}
\label{tlam}
\tlambda=G(\rho-\widetilde Tf) \quad \text{and}\quad \widetilde u = T \tlambda + \widetilde T f = T G(\rho-\widetilde Tf) + \widetilde T f.
\end{gather}
\end{proof}

The structure of exact solutions $u$ and $\lambda$ in Theorem~\ref{p:coerc} guides the discretization choices and gives rise to the MH$^2$M as a result of the discrete version of the skeletal variational problem~\eqref{e:mhhm}. This is addressed next.
\section{Discretization}
\label{sec:discrete}

\subsection{The method}
\label{ssec:method}

Consider the finite dimensional spaces
\[\Gamma_{H_\Gamma}\subset H_0^{1/2}(\btriH),\quad  \Lambda_{H_\Lambda}\subset\Lambda\quad \text{and}\quad  V_h\subset H^1(\triH),
\]
and
\[
 \tL_{H_{\Lambda}}:=\Lambda_{H_{\Lambda}}\cap\tL\quad \text{and}\quad \tVh:=V_h\cap\tVV,
\]
and denote by $\Gamma_{H_\Gamma}(\dtau)$, $\Lambda_{H_\Lambda}(\dtau)$, $\tL_{H_\Lambda}(\dtau)$,  $V_h(\tau)$ and  $\tVh(\tau)$ their restriction to $\tau\in\triH$. 

For $\tmu\in\tL$ and $q\in L^2(\Omega)$, the discrete versions of mapping $T$ and $\widetilde T$, namely  $T_h:\tL\to\tVh$ and $\widetilde T_h:L^2(\Omega)\to\tVh$, are 
\begin{equation}\label{e:Thdef}
\int_\tau\A\bgrad(T_h\tmu)\cdot\bgrad\widetilde v_h\,d\xx
=\langle\tmu,\widetilde v_h\rangle_{\dtau} \quad\text{and}\quad \int_\tau\A\bgrad(\widetilde T_hq)\cdot\bgrad\widetilde v_h\,d\xx=\int_{\tau}q\widetilde v_h\:d\xx,
\end{equation}
for all $\widetilde v_h\in\tVh$. 
Also, let $G_h:H_0^{1/2}(\btriH)\to\tL_{H_{\Lambda}}$ be the discrete operator related to $G$ (cf.~\eqref{e:Gdef}). For $\phi\in H_0^{1/2}(\btriH)$, define $\tlambda_\phi=G_h\phi$ such that 
\begin{equation}\label{e:Ghdef}
\int_\tau\A\bgrad(T_h\tlambda_\phi)\cdot\bgrad T_h\tmu_{H_{\Lambda}}\,d\xx
=\langle\tmu_{H_{\Lambda}},T_h\tlambda_\phi\rangle_\dtau
=\langle\tmu_{H_{\Lambda}},\phi\rangle_\dtau\quad\text{for all }\tmu_{H_{\Lambda}}\in\tL_{H_{\Lambda}}. 
\end{equation}
Note that the same arguments used for $G$ yield that $G_h$ is also symmetric.

\begin{remark}[Discrete local mappings]
The operators $T_h$ and $\widetilde T_h$ are well defined and  bounded. Also, $\widetilde T_h$ is bounded as $\widetilde T$ given in~\eqref{boundT}, and 
\begin{equation}
\label{bTh}
| T_h \tmu |_{H^1_\A(\tau)} \leq | \tmu |_{H^{-1/2}(\dtau)}\quad\text{for all }\tmu\in \tL\text{ and }\tau \in\triH,
\end{equation}
from \eqref{e:Thdef} and the norm definitions. On the other hand, note from~\eqref{e:Ghdef} that $G_h$ is a well-defined mapping only if $T_h$ is injective, which does not necessarily hold unless there is some kind of compatibility between spaces $\tL_{H_{\Lambda}}$ and $ \tVh$.
The details of such compatibility condition is discussed in Section~\ref{ssec:approx}.
\end{remark}


Based on Theorem~\ref{p:coerc}, we define the MH$^2$M such that $\rho_{H_\Gamma}\in\Gamma_{H_\Gamma}$ solves
\begin{equation}\label{e:mhhmh}
\langle G_h\rho_{H_{\Gamma}},\xi\rangle_\btriH =-\langle\lambda^0,\xi\rangle_\btriH +\langle G_h\widetilde T_h f,\xi\rangle_\btriH
\quad\text{for all }\xi\in \Gamma_{H_{\Gamma}},
\end{equation}
where $\lambda^0$ is given in~\eqref{u0l0}. Then, the exact solution $u$ and $\lambda$ are approximate by their discrete counterparts $u_h \in V_h$ and $\lambda_{H_\Lambda}\in\Lambda_{H_\Lambda}$, where 
\begin{equation}
\label{disc-sol}
u_h :=u_h^0+T_hG_h\rho_{H_{\Gamma}}+(I-T_hG_h)\widetilde T_h f \quad\text{and}\quad \lambda_{H_{\Lambda}} := \lambda^0+G_h(\rho_{H_{\Gamma}}-\widetilde T_h f),
\end{equation}
where $u_h^0 \in \VVRM$  is the approximate counterpart of $u_0$ in \eqref{u0l0}, i.e., 
\begin{equation}
\label{u0h}
u^0_h|_\tau = \frac{1}{|\dtau|} \int_{\dtau}  \rho_{H_{\Gamma}} \:d\xx\quad\text{for all }\tau \in\triH.
\end{equation}

\begin{remark}[Conformity of $u_h$]
Note that, in general, the MH$^2$M is a nonconforming method in $H^1(\OO)$. It becomes conforming if, for instance, we choose the (impractical) space $\tL_{H_\Lambda} = \tL$ since, in this case,
\[
u_h |_{\dtau} = {u_h^0}|_{\dtau} + {T_hG\rho_{H_{\Gamma}}}|_{\dtau}+(I-T_hG)\widetilde T_h f |_ {\dtau} = \rho_h\quad\text{for all }\tau\in \triH ,
\]
from~\eqref{e:Ghdef} with $T_h$ replacing $T$, and~\eqref{u0h}.
\end{remark}

\begin{remark}[$T_h G_h$ equivalent formulation]

From \eqref{e:Ghdef}, the relationship \eqref{e:tHident} is valid only in a weaker sense when $G_h$ replaces $G$. However, following the arguments in Remark~\ref{rem:tge} and using~\eqref{e:Thdef} and~\eqref{e:Ghdef}, we note that the M$H^2$M~\eqref{e:mhhm} is equivalent to
\begin{gather}
\label{e:rhodefah}
 \sum_{\tau\in\triH}\int_\tau \A\nabla T_hG_h\rho\cdot\nabla T_hG_h\xi
=\sum_{\tau\in\triH}\int_\tau f(T_hG_h\xi-c_\xi)
\quad\text{for all }\xi\in H_0^{1/2}(\btriH),
\end{gather} 
where $c_\xi$ is the constant given in~\eqref{e:tHident}.
\end{remark}

Before heading to the error analysis, we establish a relationship between the MH$^2$M \eqref{e:mhhm} and the MsFEM~\cite{HouWuCai99}.
%
%
\subsection{Bridging the MH$^2$M  and MsFEM}
\label{ssec:equival}

Note that equation~\eqref{e:mhhm} that defines our method has some sort of relation with the definition of the MsFEM. In fact, we show \color{black} below \color{black}  that ~\eqref{e:mhhm} yields the \emph{same trace} as the MsFEM in \emph{some} particular cases. We first consider a ``continuous version'' of the MsFEM, seeking $\rho\in H_0^{1/2}(\btriH)$ such that 
\begin{equation}\label{e:cMsFEM}
\int_\Omega\A\bgrad\E(\rho)\cdot\bgrad\E(\xi)=\int_\Omega f\E(\xi)
\quad\text{for all }\xi\in H_0^{1/2}(\btriH), 
\end{equation}
where we denote the $\A$-harmonic extension $\E:H_0^{1/2}(\btriH)\to H_0^1(\Omega)$ is such that, for $\xi\in H_0^{1/2}(\btriH)$, the extension $\E(\xi)=\xi$ on $\btriH$ and weakly solves $\div\A\bgrad\E(\xi)=0$ in each element $\tau\in\triH$. 

Note now that, for all $\tau\in\triH$ and $\xi\in H^{1/2}(\dtau)$,
\begin{equation}\label{e:identity1}
\E(\xi)=TG\xi-c_\xi,
\end{equation}
where $c_\xi$ is as in~\eqref{e:tHident}. That the identity above holds on $\dtau$ follows immediately from~\eqref{e:tHident}. Next, since $G\xi\in\tL$ then $TG\xi-c_\xi$ is $\A$-harmonic, and from uniqueness of the harmonic extension, the identity~\eqref{e:identity1} holds. As a result, we get  that solutions $\rho$  of formulations~\eqref{e:rhodefa} and~\eqref{e:cMsFEM} coincide since from~\eqref{e:identity1} and~\eqref{e:cMsFEM}
\[
\sum_{\tau\in\triH}\int_\tau\A\bgrad TG\rho\cdot\bgrad TG\xi = \int_\Omega\A\bgrad\E(\rho)\cdot\bgrad\E(\xi) = \int_\Omega f\E(\xi) = \sum_{\tau\in\triH}\int_\tau f(TG\xi-c_\xi)
\]
for all $\rho,\xi\in H_0^{1/2}(\btriH)$, and that is the same as~\eqref{e:rhodefa}.

\color{black} Discretizing \color{black}~\eqref{e:cMsFEM} using $\Gamma_{H_{\Gamma}}\subset H_0^{1/2}(\btriH)$ yields the MsFEM (c.f~\cite{HouWuCai99}) whose solution is
\[
 u_{\MsFEM} := \E(\rho_{H_{\Gamma}})  
\]
where $\rho_{H_{\Gamma}} \in \Gamma_{H_{\Gamma}}$ is the solution of~\eqref{e:cMsFEM} restricted to $\Gamma_{H_{\Gamma}}$. Note that if $T_h = T$ and $G_h = G$ then the corresponding solution $u_{h}$ of MH$^2$M relates to $u_{\MsFEM}$  as follows
\[
u_{h} = u_{\MsFEM} + (I-TG)\widetilde Tf = u_{\MsFEM} + \widetilde Tf -\EE(\widetilde Tf)
\]
where we used~\eqref{disc-sol},~\eqref{e:identity1} and~\eqref{u0h}. It follows in particular that $u_{h}|_{\dtau} = u_{\MsFEM}|_{\dtau} = \rho_{H_{\Gamma}}$ for all $\tau\in\triH$.
\begin{remark}[Polynomial solutions]
If $\Gamma_{H_{\Gamma}} = \mathbb{P}_1(\btriH)\cap H_0^{1/2}(\btriH)$ and $\mathcal{A=\alpha\,I}$, with $\alpha\in \RR^+$, then 
\[
u_{h} - \widetilde Tf  + \EE(\widetilde Tf)=u_{\MsFEM} \in \mathbb{P}_1(\triH),
\]
where  $\mathbb{P}_k(D)$ is the space of piecewise polynomial functions of degree up to $k \geq 0$ on the set  $D =  \triH$ or $D = \btriH$.
%
In addition, if $f \in  \mathbb{P}_0(\triH)$  then $\widetilde Tf \in   \mathbb{P}_2(\triH)$ such that $ \nabla \widetilde Tf \in RT_0(\triH)$ where $RT_0(\triH)$ stands for the lowest-order Raviart-Thomas space in each $\tau\in \triH$.
\end{remark}

In practice, the discrete mapping $G_h$ is used instead of $G$ and the relation~\eqref{e:identity1} does not hold. So, the solutions of MsFEM and MH$^2$M do not coincide on element boundaries in general.


%
%
\section{Numerical analysis}
\label{sec:analysis}

This section contains the proof of the well-posedness and best approximation properties of the MH$^2$M given in~\eqref{e:rhodef} (see also~\eqref{e:mhhm},~\eqref{e:mhhmh} or~\eqref{e:rhodefah}). 

\subsection{Well-posedness}
\label{ssec:approx}

Define the bilinear forms $g_{h,\tau}: \color{black} H^{1/2}(\partial\tau)\color{black} \times \color{black} H^{1/2}(\partial\tau) \color{black}\to\RR$ for $\tau\in\triH$, and $g_h:H_0^{1/2}(\btriH)\times H_0^{1/2}(\btriH)\to\RR$ by
\begin{equation}
\label{e:ghdef}
g_{h,\tau}(\xi,\phi)=\langle G_h\xi,\phi\rangle_\dtau,
\qquad
g_h(\xi,\phi)=\sum_{\tau\in\triH}g_{h,\tau}(\xi,\phi)
\end{equation}
for $\xi$, $\phi\in H_0^{1/2}(\btriH)$. Using those notations, the MH$^2$M reads: Find  $\rho_{H_\Gamma} \in \Gamma_{H_{\Gamma}}$ such that
\begin{equation}\label{e:mhhm_d}
g_h(\rho_{H_\Gamma},\xi)=-\langle\lambda^0,\xi\rangle_\btriH+g_h(\widetilde T_hf,\xi)
\quad\text{for all }\xi\in \Gamma_{H_{\Gamma}}.
\end{equation}
We first address  the existence and uniqueness of solution for~\eqref{e:mhhm_d}. This result is established under the following conditions: 

\begin{itemize}
\item[] \emph {Assumption A}: there exists a positive constant $\beta_\tau$, independent of $H_\tau$, such that
\[
|\mu |_{H^{-1/2}(\dtau)} \leq \beta_\tau\, |T_h\mu |_{H^1_{\A}(\tau)} \quad \text{for all } \mu \in  \widetilde\Lambda_{H_\Lambda} \text{ and }\tau\in\triH;
\]
\item[] \emph {Assumption B}: there exists a positive constant $\alpha_\tau$, independent of $H_\tau$, such that
\[
|\xi|_{H^{1/2}(\dtau)} \leq \alpha_\tau\, |G_h\xi |_{H^{-1/2}(\dtau)} \quad \text{for all } \xi \in \Gamma_{H_\Gamma}\cap\tH^{1/2}(\dtau) \text{ and }\tau\in\triH.
\]
\end{itemize}

Let $\alpha_{\max}:=\max\{\alpha_\tau :\,\tau\in\triH\}$ and $\beta_{\max}:=\max\{\beta_\tau:\tau\in\triH\}$. 

\begin{remark}[Boundeness of $G_h$]
A first consequence of Assumption A is that the mapping $G_h$ is bounded. Indeed, let \color{black} $\xi\in H^{1/2}_0(\btriH)$ \color{black} and select $\mu := G_h\xi\in  \widetilde\Lambda_{H_\Lambda}$ in Assumption A. Then,  from Lemma~\ref{l:ids}, item $ (iii)$
\begin{gather*}
\beta_\tau^{-2} |G_h\xi|_{H^{-1/2}(\dtau)}^2
\le|T_hG_h\xi|_{H_\A^1(\tau)}^2
=\langle G_h\xi,T_hG_h\xi\rangle_\dtau \\
=(G_h\xi,\xi)_\dtau
\le|G_h\xi|_{H^{-1/2}(\dtau)}|\xi|_{H^{1/2}(\dtau)}.
\end{gather*}
Then,
\begin{gather}
\label{bGh}
|G_h\xi|_{H^{-1/2}(\dtau)} \leq \beta_\tau^2|\xi|_{H^{1/2}(\dtau)}\quad\text{and}\quad |G_h\xi|_{\Lambda} \leq \beta_{\max}^2|\xi|_{H^{1/2}(\btriH)}.
\end{gather}
\end{remark}

The MH$^2$M is well-posed under Assumptions A and B. This follows next.

\begin{theorem}\label{l:coerch}
Let $g_{h,\tau}$, $g_h$ be defined by~\eqref{e:ghdef} and assume that Assumptions A and B hold. Then, we have the following coercivity results: 
\begin{gather*}
g_{h,\tau}(\txi_{H_{\Gamma}},\txi_{H_\Gamma})
\ge(\alpha_\tau\beta_\tau)^{-2}|\txi_{H_{\Gamma}}|_{H^{1/2}(\dtau)}^2
\quad\text{for all }\txi_{H_{\Gamma}}\in\Gamma_{H_\Gamma}\cap\tH^{1/2}(\dtau),
\\
g_h(\xi_{H_{\Gamma}},\xi_{H_{\Gamma}})\ge(\alpha_{\max}\beta_{\max})^{-2}|\xi_{H_{\Gamma}}|_{H^{1/2}(\btriH)}^2
\quad\text{for all }\xi_{H_{\Gamma}}\in\Gamma_{H_{\Gamma}}.
\end{gather*}
Moreover, the following continuity results hold:
\begin{gather*}
g_{h,\tau}(\xi_{H_{\Gamma}},\rho_{H_{\Gamma}})\le  \beta_\tau^2|\xi_{H_{\Gamma}}|_{H^{1/2}(\dtau)}|\rho_{H_{\Gamma}}|_{H^{1/2}(\dtau)}, 
\\
g_h(\xi_{H_{\Gamma}},\rho_{H_{\Gamma}})
\le \beta_{\max}^2|\xi_{H_{\Gamma}}|_{H^{1/2}(\triH)}|\rho_{H_{\Gamma}}|_{H^{1/2}(\triH)},
\end{gather*}
for all $\xi_{H_{\Gamma}}$, $\rho_{H_{\Gamma}}\in\Gamma_{H_{\Gamma}}$.
Then, the method \eqref{e:mhhm_d} is well-posed and
\[ 
|\rho_{H_\Gamma}|_{H^{1/2}(\btriH)} \leq (\alpha_{\max}\beta_{\max})^2\Big(|\lambda^0|_{\Lambda} +\beta_{\max}^2|\widetilde T_h f |_{H^{1/2}(\btriH)}\Big).
\]
\end{theorem}
\begin{proof}
Fix $\txi_{H_{\Gamma}}\in\Gamma_{H_{\Gamma}}\cap\tH^{1/2}(\dtau)$. The local coercivity holds since, from Assumptions $A$ and $B$, and the definition of mapping $G_h(\cdot)$ in \eqref{e:Ghdef}
\begin{gather*}
|\txi_{H_{\Gamma}}|_{H^{1/2}(\dtau)}^2 \le\alpha_\tau^2\beta_\tau^2|T_hG_h\txi_{H_{\Gamma}}|_{H_\A^1(\tau)}^2
=\alpha_\tau^2\beta_\tau^2\langle G_h\txi_{H_{\Gamma}},T_hG_h\txi_{H_{\Gamma}}\rangle_\dtau
=\alpha_\tau^2\beta_\tau^2 \,g_{h,\tau}(\txi_{H_{\Gamma}},\txi_{H_{\Gamma}}).
\end{gather*}
Next, take $\xi_{H_{\Gamma}}\in\Gamma_{H_{\Gamma}}$. Then $\xi_{H_{\Gamma}}=\txi_{H_{\Gamma}}+\xi_{H_{\Gamma}}^0$, where $\txi_{H_{\Gamma}}|_\dtau\in\tH^{1/2}(\dtau)$ and $\xi_{H_{\Gamma}}^0|_\dtau$ is constant for each $\tau\in\triH$. So, since $G_h(\cdot)$ is symmetric and its image has zero mean value on $\dtau$, it holds that 
\begin{gather*}
g_h(\xi_{H_{\Gamma}},\xi_{H_{\Gamma}})
=\sum_{\tau\in\triH}g_{h,\tau}(\xi_{H_{\Gamma}},\xi_{H_{\Gamma}})
=\sum_{\tau\in\triH}g_{h,\tau}(\txi_{H_{\Gamma}},\txi_{H_{\Gamma}}) \\
\ge\sum_{\tau\in\triH}(\alpha_\tau\beta_\tau)^{-2}|\txi_{H_{\Gamma}}|_{H^{1/2}(\tau)}^2
\ge(\alpha_{\max}\beta_{\max})^{-2}|\xi_{H_{\Gamma}}|_{H^{1/2}(\btriH)}^2.
\end{gather*}

To show continuity, we use Lemma~\ref{l:ids}, item $(iii)$ and \eqref{bGh} to get
\[
g_{h,\tau}(\xi_{H_{\Gamma}},\rho_{H_{\Gamma}})
\le |G_h\xi_{H_{\Gamma}}|_{H^{-1/2}(\dtau)}|\rho_{H_{\Gamma}}|_{H^{1/2}(\dtau)}
\le \beta_\tau^2|\xi_{H_{\Gamma}}|_{H^{1/2}(\dtau)}|\rho_{H_{\Gamma}}|_{H^{1/2}(\dtau)}. 
\]
Finally,
\[
g_h(\xi_{H_{\Gamma}},\rho_{H_{\Gamma}})
\le\sum_{\tau\in\triH}\beta_\tau^2|\xi_{H_{\Gamma}}|_{H^{1/2}(\dtau)}|\rho_{H_{\Gamma}}|_{H^{1/2}(\dtau)}
\le \beta_{\max}^2|\xi_{H_{\Gamma}}|_{H^{1/2}(\triH)}|\rho_{H_{\Gamma}}|_{H^{1/2}(\triH)} 
\]
from the Cauchy-Schwartz inequality and Lemma~\ref{l:halfnorm}. Existence and uniqueness of solution of~\eqref{e:mhhm_d} follow from the Lax-Milgram lemma, and 
\begin{gather*}
|\rho_{H_\Gamma}|_{H^{1/2}(\btriH)}^2 \leq (\alpha_{\max}\beta_{\max})^2 g_h(\rho_{H_\Gamma},\rho_{H_\Gamma}) =  (\alpha_{\max}\beta_{\max})^2\Big(-\langle\lambda_0,\rho_{H_\Gamma}\rangle_{\btriH} + g_h(\widetilde T_h f, \rho_{H_\Gamma})_\btriH \Big)  \\ 
\leq  (\alpha_{\max}\beta_{\max})^2\Big(|\lambda^0|_{\Lambda} + \beta_{\max}^2|\widetilde T_h f |_{H^{1/2}(\btriH)}\Big)|\rho_{H_\Gamma}|_{H^{1/2}(\btriH)},
\end{gather*}
and the result follows.
\end{proof}

Given the spaces $\widetilde\Lambda_{H_{\Lambda}}$, $\Gamma_{H_{\Gamma}}$ and $\widetilde V_{h}$, it may be difficult to verify directly whether Assumptions $A$ and $B$ are valid, and \color{black} we  adapt the standard approach based on Fortin operators acting on these finite-dimensional spaces to our setting in order to ease the proof. \color{black} The upshot is that it clarifies in what sense the spaces $\Lambda_{H_{\Lambda}} $ and $ V_{h}$, and the spaces $\Lambda_{H_{\Lambda}} $ and $\Gamma_{H_{\Gamma}}$ must be compatible to satisfy Assumptions $A$ and $B$. We detail this alternative below for both assumptions.
\begin{remark}[Assumption $A$ from a Fortin operator]
\label{sA}
Assumption $A$ is closely related to a compatibility condition between the finite-dimensional spaces $V_h$ and $\Lambda_{H_{\Lambda}}$. Specifically, let $V_{h_0}\subset V_h\subset H^1(\triH)$ be a finite dimensional space such that:
\begin{itemize}
%
\item[--] \color{black} there exists a mapping $\pi_V : H^1(\triH)\rightarrow V_{h_0}$ such that, for all  $v\in H^1(\triH)$ and $\tau\in\triH$, $\pi_V(v)$ satisfies \color{black}
%
\begin{equation}
\label{piL}
\int_{\dtau}\mu\,\pi_V(v)\,d\xx=\int_{\dtau}\mu\,v\,d\xx\quad\text{for all }\mu\in\Lambda_{H_{\Lambda}},
\qquad
\color{black} |\pi_V(v)|_{H^1_\A(\tau)} \leq \beta_\tau\, |v|_{H^1_\A(\tau)}, \color{black}
\end{equation}
where $\beta_\tau$ is a positive constant independent of mesh parameters.\color{black}
\end{itemize}

To see that~\eqref{piL} implies Assumption~A, take $\mu \in \widetilde\Lambda_{H_{\Lambda}}$ and note that $\pi_V(\tv)\in\tVhz:=V_{h_0}\cap\tH^1(\triH)$ for all $\tv\in\tH^1(\triH)$. Then, from the definition of $|\cdot|_{H^{-1/2}(\dtau)}$ in~\eqref{norm12d},~\eqref{piL}, the definition of $T_h$ operator, and Cauchy--Schwartz inequality, we get
\color{black}
\begin{gather*}
|\mu |_{H^{-1/2}(\dtau)} =  \sup_{\tphi\in\tH^{1/2}(\dtau)}\frac{\langle\mu,\tphi\rangle_\dtau}{|\tphi|_{H^{1/2}(\dtau)}}  =  \sup_{\tv\in\tH^{1}(\tau)}\frac{\langle\mu,\tv\rangle_\dtau}{|\tv|_{H^1_\A(\tau)}} \leq \beta_\tau  \sup_{\tv\in\tH^{1}(\tau) }\frac{\langle\mu,\pi_V(\tv)\rangle_\dtau}{|\pi_V(\tv)|_{H^{1}_A(\tau)}}   \\
\leq \beta_\tau   \sup_{\tv_h\in \widetilde V_{h_0}(\tau)}\frac{\langle\mu,\tv_h \rangle_\dtau}{|\tv_h|_{H^{1}_\A(\tau)}} =  \beta_\tau  \sup_{\tv_h\in \widetilde V_{h_0}(\tau)}\frac{\int_{\tau} \A\nabla T_h\mu\cdot \nabla\tv_h \,d\xx}{|\tv_h|_{H^{1}_\A(\tau)}}  
= \beta_\tau\,  |  T_h\mu |_{H^1_\A(\tau)}.
\end{gather*}
\color{black}
Note that Assumption $A$ holds with the same constant $\beta_\tau$ if $V_h$ replaces $V_{h_0}$ in~\eqref{piL}.
\end{remark}

\color{black}
The existence of a mapping $\pi_V(\cdot)$ satisfying~\eqref{piL}, in the case when $\dim \tL_{H_\Lambda}(\dtau) = \dim\tVhz(\tau)|_\dtau$,  is sufficient to  fulfill  Assumption $B$ provided that $\Gamma_{H_\Gamma}(\dtau)\cap\tH^{1/2}(\dtau) \subseteq \tVhz(\tau)|_\dtau$ for all $\tau\in\triH$. This is the subject of the next lemma.

\begin{lemma}
\label{comp-par}
Assume $\dim \tL_{H_\Lambda}(\dtau) = \dim\tVhz(\tau)|_\dtau$. If there exists a mapping $\pi_V(\cdot)$ satisfying~\eqref{piL} and 
\[
\Gamma_{H_{\Gamma}}(\dtau)\cap\tH^{1/2}(\dtau)\subseteq\tVhz(\tau)|_\dtau\quad \text{for all }\tau\in\triH,
\] 
then Assumption $B$ holds with $\alpha_\tau=\beta_\tau$.
\end{lemma}
\begin{proof}
First note that $\pi_V$ is uniquely defined since $\dim\tL_{H_\Lambda}(\dtau)=\dim\tVhz(\tau)|_\dtau$. Let $\xi\in\Gamma_{H_{\Gamma}}(\tau)\cap\tH^{1/2}(\dtau)$ and $v \in \tVhz(\tau)$  be such that $\xi |_{\partial\tau}  = v|_{\partial\tau}$ 
for all $\tau\in\triH$. Note that from the definition of $\pi_V(\cdot)$ in~\eqref{piL} and \eqref{e:Ghdef} (with $T$ instead of $T_h$), for all $\mu \in \widetilde\Lambda_{H_{\Lambda}}$, it holds that 
\begin{gather*}
\int_{\dtau} \mu\,\xi \,d\xx = \int_{\dtau} \mu\, v \,d\xx = \int_{\dtau} \mu\, T G_h v \,d\xx = \int_{\dtau} \mu\, \pi_V(T G_h v) \,d\xx\,.
\end{gather*}
Then, the above equality is also valid for all $\mu\in\Lambda_{H_\Lambda}$ and~\eqref{piL} yields 
\begin{gather}
\label{vpiv}
\int_{\dtau} \mu\, (\xi-\pi_V(T G_h v) ) \,d\xx = 0 \, \Rightarrow \,v|_{\partial\tau} = \pi_V(T G_h v)|_{\partial\tau}\in\tVhz(\tau)\quad \text{for all }\tau\in\triH.
\end{gather}
Consequently, we get from \eqref{vpiv}, the stability of $\pi_V(\cdot)$ in \eqref{piL}, Lemma~\ref{l:ids}$-(i)$ and the stability of $T$ in~\eqref{boundT}
\begin{equation*}
|\xi|_{H^{1/2}(\dtau)}=
|\pi_V(TG_h v)|_{H^{1/2}(\dtau)}
\leq\beta_\tau\,|G_hv|_{H^{-1/2}(\dtau)}
=\beta_\tau\,|G_h\xi|_{H^{-1/2}(\dtau)},
\end{equation*}
which corresponds to Assumption $B$ with $\alpha_\tau=\beta_\tau$.
\end{proof}

More generally, the validity of Assumption~B is closely connected to the existence of a Fortin operator satisfying the appropriate stability and commuting properties.
\color{black}

\begin{remark}[Assumption $B$ from a Fortin operator]
\label{sB}
Assumption $B$ is related to a compatibility condition between the spaces $\Lambda_{H_{\Lambda}}$ and $\Gamma_{H_{\Gamma}}$. Specifically, let $\Lambda_{H_{\Lambda_0}}\subset\Lambda_{H_\Lambda}\subset\Lambda$ be a finite-dimensional space such that

\begin{itemize}
\item[--] there exists a mapping $\pi_\Lambda : \color{black} \tL \rightarrow \tL_{H_{\Lambda_0}} \color{black} $ and a positive constant $\alpha_\tau$,  independent of mesh parameters, such that for all \color{black} $\tmu\in\tL$ \color{black} and $\tau\in\triH$, it follows that
\begin{equation}
\label{piGB}
\begin{gathered}
 \int_{\dtau} \pi_{\Lambda}(\color{black} \tmu \color{black})\, \xi \,d\xx = \langle \color{black}\tmu\color{black}, \xi \rangle_{\dtau} \quad\text{for all }\xi \in \Gamma_{H_\Gamma}, \\
 |\pi_\Lambda(\color{black}\tmu\color{black})|_{H^{-1/2}(\dtau)} \leq\alpha_\tau\, |\color{black}\tmu\color{black} |_{H^{-1/2}(\dtau)}.
 \end{gathered}
\end{equation}
\end{itemize}

Assumption $B$ follows from \eqref{piGB}. Indeed, take $\xi \in \Gamma_{H_{\Gamma}}(\tau)\cap\tH^{1/2}(\dtau)$ for all $\tau\in\triH$ and note that $\pi_\Lambda(\tmu)\in\tL_{H_{\Lambda_{0}}}:=\Lambda_{H_{\Lambda_0}}\cap\tL$ for all $\tmu\in\tL$. Next, use the characterization $|\cdot |_{H^{1/2}(\dtau)}$ in~\eqref{e:id3},~\eqref{piGB}, the definition of $T$ and its stability~\eqref{boundT}, Cauchy-Schwartz inequality, and Lemma~\ref{l:ids}$-(i)$, to obtain
\begin{gather*}
|\xi |_{H^{1/2}(\dtau)}=  \sup_{\tmu\in \tH^{-1/2}(\dtau)}\frac{\langle\tmu,\xi\rangle_\dtau}{|\tmu|_{H^{-1/2}(\dtau)}} \leq \alpha_\tau   \sup_{\tmu\in \tH^{-1/2}(\dtau)}\frac{\langle \pi_\Lambda(\tmu),\xi\rangle_\dtau}{|\pi_\Lambda(\tmu)|_{H^{-1/2}(\dtau)}} \leq \alpha_\tau   \sup_{\tmu_{H_\Gamma}\in \tL_{H_{\Lambda_0}}(\dtau)}\frac{\langle \tmu_{H_\Gamma},\xi\rangle_\dtau}{|\tmu_{H_\Gamma}|_{H^{-1/2}(\dtau)}} \\
= \alpha_\tau   \sup_{\tmu_{H_\Gamma}\in \tL_{H_{\Lambda_0}}(\dtau)}\frac{\int_{\tau} \A\nabla T G_h\xi \cdot \nabla T \tmu_{H_\Gamma} \,d\xx}{|\tmu_{H_\Gamma}|_{H^{-1/2}(\dtau)}} \leq  \alpha_\tau\,  |  T G_h\xi |_{H^1_\A(\tau)} = \alpha_\tau\,  |  G_h\xi |_{H^{-1/2}(\tau)} .
\end{gather*}
In addition, Assumptions $B$ holds with the same constant  $\alpha_\tau$ if $\Lambda_{H_{\Lambda}}$ replaces $\Lambda_{H_{\Lambda_0}}$  in  \eqref{piGB}.
\end{remark}

\subsection{Best approximation}
\label{sec:best} 

We start by estimating the consistency error $G-G_h$ based on the first Strang Lemma~\cite{MR2050138}. In what follows, we denote by $C$ positive constants independent of the mesh parameters, which can change at each occurrence.

\begin{lemma}\label{l:strang}
Under the assumptions of Theorem~\ref{l:coerch}, for all $\phi\in H^{1/2}(\btriH)$, it holds that 
\[
|G\phi-G_h\phi|_\Lambda
\le E(G\phi), 
\]
where
\begin{equation}\label{e:Edef}
E(G\phi):=(1+\beta_{\max}^2)\,\inf_{\tmu_{H_{\Lambda}}\in\tL_{H_{\Lambda}}}
|G\phi-\tmu_{H_{\Lambda}}|_\Lambda + \beta_{\max}^2|(T-T_h)G\phi|_{H^{1/2}(\btriH)}.
\end{equation}
\end{lemma}
\begin{proof}
Given $\phi\in H^{1/2}(\btriH)$, let $\tlambda:=G\phi$ and $\tlambda_{H_{\Lambda}}:=G_h\phi$. Then, by definition of mapping $G$ and $G_h$, for $\tmu_{H_{\Lambda}}\in\tL_{H_{\Lambda}}$, 
\begin{gather}
\label{consist}
\langle \tmu_{H_{\Lambda}},T\tlambda\rangle_\btriH
=\langle\tmu_{H_{\Lambda}},T_h \tlambda_{H_{\Lambda}}\rangle_\btriH
=\langle\tmu_{H_{\Lambda}},\phi\rangle_\btriH. 
\end{gather}
Hence, from Assumption $A$ and~\eqref{consist} 
\begin{equation*}
\begin{aligned}
\beta_{\max}^{-2}|\tlambda_{H_{\Lambda}}-\tmu_{H_{\Lambda}}|_\Lambda^2 
&\leq \langle\tlambda_{H_{\Lambda}}-\tmu_{H_{\Lambda}},T_h(\tlambda_{H_{\Lambda}}-\tmu_{H_{\Lambda}})\rangle_\btriH \nonumber
\\
&=\langle\tlambda_{H_{\Lambda}}-\tmu_{H_{\Lambda}},T_h(\tlambda-\tmu_{H_{\Lambda}})\rangle_\btriH+\langle\tlambda_{H_{\Lambda}}-\tmu_{H_{\Lambda}},T_h(\tlambda_{H_{\Lambda}}-\tlambda)\rangle_\btriH \nonumber
\\
&=\langle\tlambda_{H_{\Lambda}}-\tmu_{H_{\Lambda}},T_h(\tlambda-\tmu_{H_{\Lambda}})\rangle_\btriH+\langle\tlambda_{H_{\Lambda}}-\tmu_{H_{\Lambda}},(T-T_h)\tlambda\rangle_\btriH \nonumber
\\
&\le |\tlambda_{H_{\Lambda}}-\tmu_{H_{\Lambda}}|_\Lambda \Big( |\tlambda-\tmu_{H_{\Lambda}}|_\Lambda + | (T-T_h)\tlambda |_{{H^{1/2}(\btriH)}}\Big) \label{aux1}
\end{aligned}
\end{equation*}
where we also used Lemma~\ref{l:ids}$-(iii,iv)$ (with $T$ replaced by $T_h$) and~\eqref{bTh}. From the triangle inequality and inequality above we get
\[
|\tlambda-\tlambda_{H_{\Lambda}}|_\Lambda \leq |\tlambda-\tmu_{H_{\Lambda}}|_\Lambda + |\tlambda_{H_{\Lambda}}-\tmu_{H_{\Lambda}}|_\Lambda \leq \Big(1+\beta_{\max}^{2}\Big) |\tlambda-\tmu_{H_{\Lambda}}|_\Lambda + \beta_{\max}^{2} | (T-T_h)\tlambda |_{{H^{1/2}(\btriH)}}
\]
and the result follows. 
\end{proof}
Owing to the previous results, the next theorem shows that the method yields best approximation results.

\begin{theorem}\label{t:errorEstimates}
Assume that the conditions $A$ and $B$ are valid, and let $(u,\rho,\lambda)\in H^1(\mathcal{T}_H)\times H_0^{1/2}(\mathcal{T}_H)\times\Lambda$ solve~\eqref{e:weak-hybridH} and $\rho_{H_\Gamma}\in\Lambda_{H_\Lambda}$ solves~\eqref{e:mhhmh}, and  $(u_h,\lambda_{H_\Lambda})\in V_h\times\Lambda_{H_\Lambda}$ be given in~\eqref{disc-sol}. Then,
\begin{equation}
\label{e:rho0}
\begin{gathered}
|\rho-\rho_{H_{\Gamma}}|_{H^{1/2}(\btriH)}
\le C\,\Big( \inf_{\phi_{H_{\Gamma}}\in\Gamma_{H_{\Gamma}}} |\rho-\phi_{H_{\Gamma}}|_{H^{1/2}(\btriH)} + |(\widetilde T-\widetilde T_h)f|_{H^{1/2}(\btriH)}   +E(\tlambda)\Big),
\end{gathered}
\end{equation}
and
\begin{equation}
\label{e:lambdaest}
\begin{aligned}
|\lambda-\lambda_{H_\Lambda}|_\Lambda
&\leq C\,\Big(|\rho-\rho_{H_\Gamma}|_{H^{1/2}(\btriH)} +|(\widetilde T-\widetilde T_h)f|_{H^{1/2}(\btriH)}   +E(\tlambda)\Big),
\\
|u-u_h|_{H_\A^1(\triH)}&\le
|\lambda-\lambda_{H_\Lambda}|_\Lambda
+|(T-T_h)\tlambda + (\widetilde T-\widetilde T_h)f|_{H_\A^1(\triH)}.
\end{aligned}
\end{equation}
Moreover, the following weak continuity holds: 
\begin{equation}\label{e:wc}
\langle\mu_{H_\Lambda},u_h-\rho_{H_\Gamma}\rangle_\btriH=0
\quad\text{for all }\mu_{H_\Lambda}\in\Lambda_{H_\Lambda},
\end{equation}
and the discrete flux $\lambda_{H_\Lambda}$ respects the local equilibrium constraint
\begin{equation}\label{e:mass}
\int_{\dtau} \lambda_{H_\Lambda} \,d\xx = \int_\tau f \,d\xx\quad\text{for all }\tau\in\triH.
\end{equation}
\end{theorem}
\begin{proof}
We gather from the triangle inequality that, for an arbitrary $\phi_{H_{\Gamma}}\in\Gamma_{H_{\Gamma}}$, 
\begin{equation}\label{e:rho1}
|\rho-\rho_{H_\Gamma}|_{H^{1/2}(\btriH)}
\le |\rho-\phi_{H_\Gamma}|_{H^{1/2}(\btriH)}+|\phi_{H_\Gamma}-\rho_{H_\Gamma}|_{H^{1/2}(\btriH)}, 
\end{equation}
and from Theorem~\ref{l:coerch},~\eqref{e:mhhm}, Lemma~\eqref{l:strang} and~\eqref{bGh},
\begin{equation*}
\begin{aligned}
(\alpha_{max}&\beta_{max})^{-2}|\phi_{H_{\Gamma}}-\rho_{H_{\Gamma}}|_{H^{1/2}(\btriH)}^2 \\
&\leq g_h(\phi_{H_{\Gamma}}-\rho_{H_{\Gamma}},\phi_{H_{\Gamma}}-\rho_{H_{\Gamma}}) \\
&= g_h(\phi_{H_{\Gamma}}-\rho,\phi_{H_{\Gamma}}-\rho_{H_{\Gamma}}) - g_h(\rho_{H_{\Gamma}},\phi_{H_{\Gamma}}-\rho_{H_{\Gamma}})+  g_h(\rho,\phi_{H_{\Gamma}}-\rho_{H_{\Gamma}})\\
&\quad + g(\rho,\phi_{H_{\Gamma}}-\rho_{H_{\Gamma}})- g(\rho,\phi_{H_{\Gamma}}-\rho_{H_{\Gamma}})
\\
&=\langle G_h(\phi_{H_{\Gamma}}-\rho),\phi_{H_{\Gamma}}-\rho_{H_{\Gamma}}\rangle_\btriH
+\langle (G_h-G)\rho,\phi_{H_{\Gamma}}-\rho_{H_{\Gamma}}\rangle_\btriH \\
&\quad + \langle (G\widetilde T  - G_h\widetilde T_h) f,\phi_{H_{\Gamma}}-\rho_{H_{\Gamma}}\rangle_\btriH\\ 
&=\langle G_h(\phi_{H_{\Gamma}}-\rho),\phi_{H_{\Gamma}}-\rho_{H_{\Gamma}}\rangle_\btriH
+\langle (G_h-G)\rho,\phi_{H_{\Gamma}}-\rho_{H_{\Gamma}}\rangle_\btriH \\ 
&\quad + \langle (G  - G_h) \widetilde T f,\phi_{H_{\Gamma}}-\rho_{H_{\Gamma}}\rangle_\btriH+ \langle G_h( \widetilde T -  \widetilde T_h)f,\phi_{H_{\Gamma}}-\rho_{H_{\Gamma}}\rangle_\btriH \\
&\leq   \Big[\beta_{max}^{2}\Big( |\rho-\phi_{H_{\Gamma}}|_{H^{1/2}(\btriH)} + \vert (\widetilde T-\widetilde T_h) f \vert_{H^{1/2}(\btriH)}\Big) + E(G(\rho- \widetilde T f))  \Big] | \phi_{H_{\Gamma}}-\rho_{H_{\Gamma}}|_{H^{1/2}(\btriH)} \\
\end{aligned}
\end{equation*}
and the  result~\eqref{e:rho0}  follows from \eqref{e:rho1} and the above inequality recalling that $\tlambda = G(\rho- \widetilde T f)$.

Also, from Lemma~\ref{l:strang} and~\eqref{bGh} 
\begin{equation}
\label{e:rho3}
\begin{gathered}
|(G\rho-G_h\rho_{H_{\Lambda}}-(G\widetilde T-G_h\widetilde T_h)f|_\Lambda
\le|(G-G_h)(\rho-\widetilde T f)|_\Lambda+|G_h(\widetilde T_h f-\widetilde T f) |_\Lambda + |G_h(\rho-\rho_{H_{\Lambda}}) |_\Lambda
\\
\le E(\tlambda) + \beta_{\max}^2\Big(|(\widetilde T-\widetilde T_h)f |_{H^{1/2}(\btriH)} + |\rho-\rho_{H_{\Lambda}} |_{H^{1/2}(\btriH)}\Big),
\end{gathered}
\end{equation}
and the first estimate in~\eqref{e:lambdaest} follows since 
\begin{gather}
\lambda-\lambda_{H_{\Lambda}}
=\tlambda-\tlambda_{H_{\Lambda}}= G\rho-G_h\rho_{H_{\Lambda}}- (G\widetilde T - G_h\widetilde T_h)f . \label{elambda}
\end{gather}

Next, using 
\begin{gather}
u-u_h=u^0-u_h^0+\widetilde u-\widetilde u_h \label{euh},
\end{gather}
the second estimate in~\eqref{e:lambdaest} follows from 
\begin{equation*}
\tilde{u}-\tilde{u}_h=T\tlambda-T_h\tlambda_{H_{\Lambda}}+\widetilde{T}f-\widetilde{T}_h f=T_h(\tlambda-\tlambda_{H_{\Lambda}})+(T-T_h)\tlambda+(\widetilde{T}-\widetilde{T}_h)f
\end{equation*}
and using the stability of  $T_h$ in \eqref{bTh} to get
\begin{equation*}
\vert T_h(\lambda-\lambda_{H_{\Lambda}})\vert_{H_\A^1(\triH)} = \vert T_h(\tlambda-\tlambda_{H_{\Lambda}})\vert_{H_\A^1(\triH)}
\leq\vert\tlambda-\tlambda_{H_{\Gamma}}\vert_{\Lambda}=\vert\lambda-\lambda_{H_{\Gamma}}\vert_{\Lambda}.
\end{equation*}
To show~\eqref{e:wc}, for $\tmu_{H_{\Lambda}}\in\tL_{H_{\Lambda}}$ and using the definition of the mapping $G_h$ in \eqref{e:Ghdef}, we get 
\[
\langle\tmu_{H_\Lambda},u_h\rangle_\btriH
=\langle\tmu_{H_\Lambda},\widetilde u_h\rangle_\btriH
=\bigl\langle\tmu_{H_\Lambda},T_hG_h\rho_{H_{\Gamma}}+(I-T_hG_h)\widetilde T_hf\bigr\rangle_\btriH
=\langle\tmu_{H_\Lambda},\rho_{H_\Gamma}\rangle_\btriH.
\]
Next, for $\mu_{H_{\Lambda}}\in\Lambda^0$ it  follows from~\eqref{u0h}, 
\[
\langle\mu_{H_\Lambda},u_h\rangle_\btriH
=\langle\mu_{H_\Lambda},u^0_h\rangle_\btriH
=\langle\mu_{H_\Lambda},\rho_{H_\Gamma}\rangle_\btriH,
\]
and then~\eqref{e:wc} holds. Finally, the equilibrium condition~\eqref{e:mass} follows from~\eqref{u0l0} since
\[
\int_{\dtau} \lambda_{H_\Lambda} \,d\xx = \int_{\dtau} \lambda^{0} \,d\xx = \int_\tau f \,d\xx\quad\text{for all }\tau\in\triH.
\]
\end{proof}
%
%
\section{Compatible finite element spaces}
\label{sec:compatible}

This section provides examples of compatible two dimensional finite element spaces satisfying  Assumptions $A$ and $B$. To this end, we introduce partitions of the elements and faces, which
could be different for different elements. 

\subsection{Element and face partitions}
\label{ssec:elefacepart}
Consider the following \color{black} family of regular partitions: \color{black}
\begin{itemize}
\item $\triho(\tau)$ (baseline partition): conforming triangulation of $\tau\in\triH$;
\item $\trih(\tau)$: conforming triangulation of $\tau\in\triH$, with diameter $h$, obtained by refining $\triho(\tau)$;
\item $\triho$: union of all $\triho(\tau)$ yielding a triangulation of $\Omega$ (not globally conforming in general)
  \[
  \triho:=\cup_{\tau\in\triH}\triho(\tau);
  \]
\item $\trih$: union of all $\trih(\tau)$ yielding a triangulation of $\Omega$ (not globally conforming in general)
  \[
  \trih:=\cup_{\tau\in\triH}\trih(\tau)\quad \text{with diameter } h;
  \]
\item $\EE(\dtau)$: set of faces associated with $\tau\in\triH$
  \[
  \EE(\dtau):=\{F\subset\dtau:\,F\text{ is a face of }\tau\};
  \]
\item $\EE$: set of faces associated with the partition $\triH$
  \[
  \EE:=\cup_{\tau\in\triH}\EE(\dtau);
  \]
   \color{black}
 \item  $\EE_{H_{\Gamma}}(\dtau)$: conforming partition of $\dtau$, which is assumed locally
quasi-uniform, namely neighboring boundary edges have comparable
diameters;
 \color{black}
\item $\EE_{H_{\Lambda_0}}(\dtau)$ (baseline partition): conforming partition of $\dtau$. Moreover, if two elements share a face, then the corresponding face triangulation is identical;
\item  $\EE_{H_{\Lambda}}(\dtau)$: conforming partition of $\dtau$, obtained by refining $\EE_{H_{\Lambda_0}}(\dtau)$;
\item $\EE_{H_{\Lambda_0}}$ (baseline partition): union of all baseline face partitions
  \[
  \EE_{H_{\Lambda_0}}:=\cup_{\tau\in\triH}\EE_{H_{\Lambda_0}}(\dtau);
  \]
\item $\EE_{H_{\Lambda}}$: union of all $\EE_{H_{\Lambda}}(\dtau)$ yielding  a partition of $ \EE$
  \[
  \EE_{H_{\Lambda}}:=\cup_{\tau\in\triH}\EE_{H_{\Lambda}}(\dtau)  \quad\text{with diameter } H_{\Lambda};
  \]
\item \color{black} $\EE_{H_{\Gamma}}$: union of all $\EE_{H_{\Gamma}}(\dtau)$ yielding  a partition of $\EE$
  \[
  \EE_{H_{\Gamma}}:=\cup_{\tau\in\triH}\EE_{H_{\Gamma}}(\dtau)\quad\text{with diameter } H_{\Gamma}. 
  \color{black}
   \]
\end{itemize}

We assume that the triangulations of faces and elements are general, but related to each other. Specifically, given $\tau\in\triH$, the baseline element \color{black}and edge partitions $\triho(\tau)$ and $\EE_{H_{\Lambda_0}}(\dtau)$, respectively, are such that they satisfy: \color{black}
~\\
\begin{itemize}
\item[(M1)] Given $\tau\in\triH$ and $F \in \EE_{H_{\Lambda}}(\dtau)$, there exist  two elements $\kappa_1, \kappa_2 \in \triho(\tau)$ such that $(\partial\kappa_1 \cap \dtau) \cup (\partial\kappa_2 \cap \dtau) = F$;
\color{black}
\item[(M2)] Given  $F \in \EE_{H_{\Gamma}}(\dtau)$, there exist  two elements $\frak{f}_1, \frak{f}_1 \in \EE_{H_{\Lambda_0}}(\dtau)$ such that $\frak{f}_1 \cup \frak{f}_2 = F$, $\frak{f}_1 \cap \frak{f}_2$ has zero measure, and $|\frak f_j| \simeq |F|$, $j=1,2$.
\color{black}
\end{itemize}
~\\
Given $\tau\in\triH$, a practical way to construct these meshes is: 
\begin{enumerate}
\item  \color{black} Set $\EE_{H_{\Gamma}}(\dtau)$; \color{black}
\item  \color{black} Define  $\EE_{H_{\Lambda_0}}(\dtau)$ such condition $(M2)$ is valid, and $\EE_{H_{\Lambda}}(\dtau)$ as a refinement of $\EE_{H_{\Lambda_0}}(\dtau)$; \color{black}
\item Define $\triho(\tau)$ such that condition $(M1)$ is valid, and $\trih(\tau)$ as a refinement of $\triho(\tau)$.
\end{enumerate}


We next present some stable finite dimensional spaces.
%
%
\subsection{The  $\PP_{k+1}(F)\times \PP_k(F) \times \mathbb{P}_{k+1} (\tau)$ element}
\label{ssec:ppp}
Let $k\geq0$ and $\Gamma_{H_\Gamma}$ be the space of continuous piecewise polynomial functions on faces of degree up to $k+1$, i.e.,
\begin{equation}\label{GammaFEM}
\Gamma_{H_\Gamma}:=\left\{\xi\in H_0^{1/2}(\btriH)\,: \, \xi |_{F}\in\PP_{k+1}(F),\, F\in\EE _{H_{\Gamma}} \right\},
\end{equation}
and $\Lambda_{H_{\Lambda}},\,\Lambda_{H_{\Lambda_0}}$ and $\tL_{H_{\Lambda}},\,\tL_{H_{\Lambda_0}}$ be the following spaces of discontinuous piecewise polynomial functions on faces of degree up to $k$ 
\begin{equation}\label{LambdaFEM}
  \begin{gathered}
  \Lambda_{H_{\Lambda}}:=\left\{\mu\in\Lambda:\,\mu |_{F} \in\PP_k(F),\, F\in \EE_{H_{\Lambda}} \right\},
  \qquad 
  \tL_{H_{\Lambda}}:= \Lambda_{H_{\Lambda}}\cap\tL, \\
  \Lambda_{H_{\Lambda_0}}:=\left\{ \mu\in\Lambda:\,\mu |_{F} \in\PP_k(F),\,  F\in \EE_{H_{\Lambda_0}} \right\},
  \qquad 
  \tL_{H_{\Lambda_0}}:= \Lambda_{H_{\Lambda_0}}\cap\tL .
  \end{gathered}
\end{equation}

In addition, we chose the following functional spaces of degree up to $k+1$ inside the elements i.e., 
\begin{equation}\label{Vh0-one}
  \begin{gathered}
V_{h_0}:=\left\{v_{h}\in H^1(\triH):\,v_{h}|_\tau\in\mathbb{P}_{k+1}(\tau),\,\tau\in\triho\right\}, 
\qquad
\tVhz:= V_{h_0}\cap\tH^1(\triH),
\\
V_{h}:=\left\{v_{h}\in H^1(\triH):\,v_{h}|_\tau\in\mathbb{P}_{k+1}(\tau),\,\tau\in\trih\right\}, 
\qquad
\tVh:=V_h\cap\tH^{1}(\triH).   
\end{gathered}
\end{equation}

\color{black}
We are ready to present the main result of this section. 

\begin{theorem}[Well-posedness and best approximation]
\label{wellposedness-mh2m}
Assume that $(M1)$ and \color{black} $(M2)$  hold. Let $\Gamma_{H_\Gamma}$ be given in \eqref{GammaFEM}, and $\Lambda_{H_\Lambda}  $ and $V_{h_0}$  be such that $\Lambda_{H_{\Lambda_0}} \subset\Lambda_{H_\Lambda}  $ and $V_{h_0}\subset V_h$, where $\Lambda_{H_{\Lambda_0}}$ is given in \eqref{LambdaFEM} and $ V_{h_0}$  in \eqref{Vh0-one}.  Then,  the MH$^2$M \eqref{e:mhhmh} is well-posed and~\eqref{disc-sol} satisfies the estimates in Theorem~\ref{t:errorEstimates}.
\end{theorem}
\begin{proof}
\color{black} First, note that from assumption (M1) and the definition of the finite-dimensional spaces \eqref{LambdaFEM} and \eqref{Vh0-one},    there exists a Fortin operator $\Pi_V:H^1(\triH)\rightarrow V_{h_0}$ (see~\cite[Lemma 2]{BarJaiParVal20}) \color{black} satisfying 
\begin{equation}
\begin{gathered}
\label{piLG}
\int_{\dtau} \mu\, \Pi_V(v) \,d\xx = \int_{\dtau} \mu\, v \,d\xx\quad\text{for all }\mu \in \Lambda_{H_{\Lambda}}, \\      
\|\Pi_V(v)\|_{H^1(\tau)} \leq C_\tau\, \|v\|_{H^1(\tau)} \quad\text{for all }v\in  V\text{ and }\tau\in\triH,
\end{gathered}
\end{equation}
where the positive constant $C_\tau$ is $H_\tau$-independent.
Then, we define $\pi_V(\cdot)$ as the operator $\Pi_V(\cdot)$ and the first equation in \eqref{piL} is valid. 
Also, for all $v\in \tH^1(\triH)$ and $\tau\in\triH$, from~\eqref{piLG} and the Poincar\'e inequality it holds that 
\begin{equation*}
\color{black}
\begin{aligned}
|\pi_V(v)|_{H^{1}_\A(\tau)} = |\Pi_V(v)|_{H^{1}_\A(\tau)} 
&\leq  a_{max}^{1/2} \|\Pi_V(v)\|_{H^{1}(\tau)}  \\
&\leq  C_\tau\, a_{max}^{1/2}  \|v\|_{H^{1}(\tau)}    \\
&  \leq C_\tau\,a_{max}^{1/2}(1+C_P H_\tau)  |v|_{H^{1}(\tau)}   \\
&\leq \frac{C_\tau\,a_{max}^{1/2}(1+C_P H_\tau) }{a_{min}^{1/2}}   |v |_{H^{1}_\A(\tau)},
\end{aligned}
\color{black}
\end{equation*}
and then  $|\pi_V(v)|_{H^{1}_\A(\dtau)} \leq \beta_\tau  |v |_{H^{1}_\A(\dtau)}$  with $\beta_\tau =  \frac{C_\tau\,a_{max}^{1/2}(1+C_P H_\tau) }{a_{min}^{1/2}}$. \color{black}  Assumption A follows from Remark~\ref{sA} by taking the mapping $\pi_V(\cdot) := \Pi_V(\cdot)$ defined above.  To establish Assumption~B, we apply Remark~\ref{sB} and make use of the Fortin operator $\pi_\Lambda(\cdot)$ satisfying~\eqref{piGB} (see Lemma \ref{l0:m2} in the Appendix). Since Assumptions~A and~B are satisfied, the result follows immediately from Theorems~\ref{l:coerch} and~\ref{t:errorEstimates}.
\end{proof}
\color{black}

%

\color{black}

\color{black}
\begin{remark}[Relaxing Assumptions $(M1)$  and $(M2)$ ]
\label{rem:relaxm1}
The mesh conditions $(M1)$ and $(M2)$ are   sufficient to establish the well-posedness of MH$^2$M \eqref{e:mhhmh}, but  are not necessary in general. For example, let $k\geq 0$ be even and consider the finite element 
\[
\PP_{k+1}(F)\times\PP_{k}(F)\times\mathbb{P}_{k+1}(\tau)
\] 
supported on partitions 
\[
\EE_{H_\Gamma }(\dtau)=\EE_{H_{\Lambda}}(\dtau)
\]
consisting of one element per face, and with $\triho(\tau)$ consisting of the single simplex $\tau$, for each  $\tau \in \triH$ (see Figure~\ref{f:representativeFEM} for an illustration).  Hence, adapting the arguments of \cite[Lemma~10]{RavTho77} to our setting, there exists a positive constant $C_\tau$, independent of mesh parameters, such that for all $\tmu\in \tL_{H_\Lambda}$
\begin{equation}
\label{wpRT}
\sup_{v\in H^1(\tau)} \frac{\langle \tmu, v \rangle_{\dtau}}{\|v\|_{H^1(\tau)}} \leq C_\tau \sup_{\tv_h\in \tVhz(\tau)} \frac{\langle \tmu, \tv_h \rangle_{\dtau}}{\|\tv_h\|_{H^1(\tau)}}.
\end{equation}
Moreover, observe that
\[
\Gamma_{H_\Gamma}(\partial\tau)=V_{h_0}(\tau)|_{\partial\tau}.
\]
Since the moment matrix corresponding to the pairing
\[
\langle \tmu,\tv_h\rangle_{\partial\tau},
\qquad
\tmu\in \widetilde\Lambda_{H_\Lambda}(\partial\tau),
\quad
\tv_h\in \tVhz(\tau)|_{\partial\tau},
\]
is square and nonsingular according to \eqref{wpRT}, Lemma~\ref{comp-par} guarantees Assumption~B. Therefore, all the statements of Theorem~\ref{wellposedness-mh2m} hold for this family of elements as well.   

In the case where $k$ is odd, condition~\eqref{wpRT} fails because the local kernel contains a Legendre polynomial of degree $k+1$. Nevertheless, when homogeneous Dirichlet boundary conditions are imposed on $\dO$, the interelement continuity constraints and the prescribed boundary values prevent these local kernel functions from giving rise to a nontrivial global function. Hence, the global kernel is trivial and a global inf-sup condition still holds. As a result, the discrete problem remains well posed despite the failure of the local condition.  However, we have not been able to establish the independence of the corresponding inf-sup constant with respect to $H_\tau$, although the numerical experiments suggest that this property holds.
\end{remark}
\color{black}

\color{black}
\begin{remark}[Comparison with the Marini--Brezzi three-field method]
The element $\PP_{1}(F)\times \PP_{0}(F)\times \mathbb{P}_{1}(\tau)$,
consisting of piecewise linear continuous functions for the trace variable and piecewise constant functions for the flux variable, was first introduced in the three-field finite element method~\cite{MR1262602}, where it was shown to be well posed and optimally convergent~\cite{MR1826573}. These properties were achieved by adding stabilization terms to the standard Galerkin discretization of the continuous three-field weak formulation~\eqref{e:weak-hybridH}. Interestingly, the MH$^2$M method~\eqref{e:mhhmh}--\eqref{u0h} is also stable and optimally convergent when using the same element $\PP_{1}(F)\times \PP_{0}(F)\times \mathbb{P}_{1}(\tau)$. However, in contrast to the three-field method, no additional stabilization terms are required. Instead, stability is obtained by assuming conditions $(M1)$ and $(M2)$ on the submeshes.
\end{remark}
\color{black}
 
\begin{figure}[h]
\begin{center}
\begin{tikzpicture}[scale=0.6]
\coordinate (A1) at (0,0);
\coordinate (A2) at (4,0);
\coordinate (A3) at (2,3);
\coordinate (A4) at (6,3);

\coordinate (B1) at (7,0);
\coordinate (B2) at (11,0);
\coordinate (B3) at (9,3);
\coordinate (B4) at (13,3);
%
\coordinate (L00) at (0,0.3);
\coordinate (L01) at (4,0.3);
\coordinate (L10) at (0,0.6);
\coordinate (L11) at (2,3.6);
\coordinate (L20) at (3.8,0.1);
\coordinate (L21) at (1.8,3.1);
\coordinate (F00) at (7,0.2);
\coordinate (F01) at (11,0.4);
\coordinate (F11) at (9,3.6);
\draw[fill=gray!20,  thick ] (A1) -- (A2) -- (A3) -- cycle ;
\draw[fill=gray!20,  thick ] (B1) -- (B2) -- (B3) -- cycle ;
\draw [red, thick ](L00) -- (L01);
\draw [red, thick ](L10) -- (L11);
\draw [red, thick ](L20) -- (L21);
\draw [blue, thick ](F00) -- (F01) -- (F11) -- cycle;
\end{tikzpicture}
\end{center}
\caption{A piecewise constant function representative of the space $\Lambda_{H_{\Lambda}}$ and a continuous piecewise linear function of the space $\Gamma_{H_{\Gamma}}$ with faces and element discretized with one element.\label{f:representativeFEM} }
\end{figure}

\subsection{The  $\PP_{k+1}(F)\times \PP_{k+1}(F) \times \mathbb{P}_{k+1} (\tau)$ element}\label{ssec:kkk}

Here we use a richer space to approximate the flux variable $\lambda$ than in the previous case. Nonetheless, the strategy of proving well-posedness and approximation properties for the MH$^2$M with the element $\PP_{k+1}(F)\times \PP_{k+1}(F) \times \mathbb {P}_{k+1} (\tau)$ also follows from the existence of \color{black} bounded mappings that satisfy~\eqref{piL} and~\eqref{piGB}.\color{black} To do so, we must consider two scenarios \color{black} regarding~\eqref{piL}: \color{black}
\begin{itemize}
\item First, if $k\geq 2$ then it is sufficient to assume (M1) to guarantee the existence of a mapping $\Pi_V(\cdot)$ that satisfies~\eqref{piLG} (c.f.~\cite[Section 3.2]{BarJaiParVal20}) and then Assumption A is satisfied using the same argument used for the element $\PP_{k+1}(F)\times \PP_{k}(F) \times \mathbb{P}_{ k+1 }(\tau)$;
\item On the other hand, when $k=0$ and $k=1$, the condition (M1) must be replaced by:
\begin{itemize}
\item[]\qquad Given $F \in \EE_{H_{\Lambda}}(\dtau)$, there are three elements $\kappa_1, \kappa_2, \kappa_3 \in \triho(\tau)$ such that $ (\partial\kappa_1 \cap \dtau) \cup (\partial\kappa_2 \cap \dtau) \cup (\partial\kappa_3 \cap \dtau) = F$.
\end{itemize}
 This condition guarantees the existence of $\Pi_V(\cdot)$ that satisfies~\eqref{piLG} (c.f.~\cite[Section 3.2]{BarJaiParVal20}).
 \end{itemize}
 
 \color{black}
 Note that Lemma~\ref{l0:m2} remains valid under the same Assumption~$(M2)$, since the space $\Lambda_{H_{\Lambda_0}}$ is richer than the one considered in the previous section. Consequently, the well-posedness and best approximation results established in Theorem~\ref{wellposedness-mh2m} are also valid using again the proof strategy used for the element $\PP_{k+1}(F)\times \PP_{k}(F) \times \mathbb{P}_{ k +1 } (\tau)$.
 
\begin{remark}[Avoiding condition $(M2)$]
Interestingly, if Assumption~$(M2)$ is replaced by the simpler choice
\[
\EE_{H_{\Lambda_0}}(\dtau)=\triho(\tau)|_{\dtau},
\]
then Lemma~\ref{l0:m2} remains valid, and hence so does Theorem~\ref{wellposedness-mh2m}.
\end{remark}
 \color{black}
 
\subsection{The one-level \color{black} $\PP_{k+1}(F)\times \PP_{k}(F)$ and $\PP_{k+1}(F)\times \PP_{k+1}(F)$\color{black}}

The method nomenclature indicates that we replace $V_{h_0}$ by $H^1(\triH)$ (i.e., $T_h = T$ and $\widetilde T _h =\widetilde  T$), which corresponds to assuming that the multiscale basis functions driven by the operators $T$ and $\widetilde  T$ have a close formula (see algorithm in Section \ref{sec:basis} for details). 
%
Notably, the one-level MH$^2$M corresponds to find $\rho_{H_{\Gamma}}^{one} \in \Gamma_{H_{\Gamma}}$ such that
\begin{equation}\label{e:mhhm-one}
\langle G_h\rho_{H_{\Gamma}}^{one},\xi\rangle_\btriH =-\langle\lambda^0,\xi\rangle_\btriH +\langle G_h\widetilde T f,\xi\rangle_\btriH
\quad\text{for all }\xi\in \Gamma_{H_{\Gamma}},
\end{equation}
where $\lambda^0$ is given in \eqref{u0l0}. The approximate solution $u_h^{one}$ and $\lambda_{h}^{one}$ are
\begin{equation}
\label{disc-sol-one}
u_h^{one} :=u_h^{one,0}+T G_h\rho_{H_{\Gamma}}^{one}+(I-T G_h)\widetilde T f \quad\text{and}\quad \lambda^{one}_{H_\Lambda} := \lambda^0 + G_h(\rho_{H_{\Gamma}}^{one}-\widetilde T f),
\end{equation}
where $u_h^{one,0}$ satisfies \eqref{u0h} using $\rho_{H_{\Gamma}}^{one}$ in place. The function $u_h^{one}$ is finite dimensional although the space spam by the image of  mappings $T$ and $\widetilde T$ are not expected to be polynomial. Recalling that in the one-level case  $T_h = T$, note that 
\begin{gather*}
|\mu |_{H^{-1/2}(\dtau)} =  \sup_{\tphi\in\tH^{1/2}(\dtau)}\frac{\langle\mu,\tphi\rangle_\dtau}{|\tphi|_{H^{1/2}(\dtau)}}  =   \sup_{\substack{\tv\in\tH^{1}(\tau) \\ \tv|_{\dtau} = \tphi}}\frac{\int_{\tau} \A\nabla T\mu\cdot \nabla \tv \,d\xx}{|\tv|_{H^{1/2}(\dtau)}}  
=  |  T\mu |_{H^1_\A(\tau)},
\end{gather*}
  for all $\mu\in\widetilde\Lambda_{H_{\Lambda}}$, and then the Assumption $A$ is valid without the need for Condition $(M1)$.
 
 \color{black} 
 Since the pair of local spaces $\Lambda_{H_{\Lambda}}(\partial\tau)\times \Gamma_{H_{\Gamma}}(\partial\tau)$
coincides with that considered in the previous sections for the one-level MH$^2$M method, Assumption B also holds. Therefore, by Theorem~\ref{wellposedness-mh2m}, the one-level MH$^2$M formulation~\eqref{e:mhhm-one} is well posed, admitting the unique solution~\eqref{disc-sol-one}. Moreover, the best-approximation estimates of Theorem~\ref{t:errorEstimates} simplify to
 \color{black}
\begin{equation*}
\begin{gathered}
|\rho-\rho^{one}_{H_{\Gamma}}|_{H^{1/2}(\btriH)}
\le C\,\Big( \inf_{\phi_{H_{\Gamma}}\in\Gamma_{H_{\Gamma}}} |\rho-\phi_{H_{\Gamma}}|_{H^{1/2}(\btriH)}   +E(\tlambda)\Big), \\
|\lambda-\lambda^{one}_{H_\Lambda}|_\Lambda
\leq C\,\Big(|\rho-\rho^{one}_{H_\Gamma}|_{H^{1/2}(\btriH)}   +E(\tlambda)\Big)
\quad\text{and}\quad |u-u_h^{one}|_{H_\A^1(\triH)}\le
|\lambda-\lambda^{one}_{H_\Lambda}|_\Lambda,
\end{gathered}
\end{equation*}
where $E(\tlambda) = (1+\beta_{\max}^2)\,\inf_{\tmu_{H_{\Lambda}}\in\tL_{H_{\Lambda}}}
|\tlambda-\tmu_{H_{\Lambda}}|_\Lambda$.

\subsection{Convergence}
\label{ssec:conv}

We establish error estimates for the two-level MH$^2$M \eqref{e:mhhm_d} using the $\PP_{k+1}(F)\times \PP_k(F) \times \mathbb{P}_ {k+1} (\tau)$ element of Section \ref{ssec:ppp}. First, we recall some interpolation operators with optimal properties.
Let $w\in H^{k+2}(\triH)$ and $\A\nabla w \in H^{k+1}(\triH)\cap\Hdiv $. Then,  we have:
\begin{itemize}
\item[] There exist $\phi_{H_\Gamma} \in \Gamma_{H_\Gamma}$ and $\tmu_{H_\Lambda}\in\tL_{H_\Lambda}$ such that
\begin{gather} 
\label{interp-bound}
\vert\xi-\phi_{H_\Gamma}\vert_{H^{1/2}(\btriH)} \leq C\,H_{\Gamma}^{k+1} | w |_{H^{k+2}(\triH)}\quad\text{and}\quad \vert\tmu-\tmu_{H_\Lambda}\vert_{\Lambda} \leq C\,H_{\Lambda}^{k+1} | \A\nabla w |_{H^{k+1}(\triH)} 
\end{gather}
\end{itemize}
where $\xi:=w\,|_{\dtau}$ and $\mu\,|_{\dtau}:=\A\nabla  w\cdot \nn^\tau\,|_{\dtau}$ for all $\tau\in\triH$. For example, we may choose $\phi_{H_\Gamma}$ as  the standard Lagrange interpolator on $\PP_k(\triH)$~\cite{MR2050138} and  $\tmu_{H_\Lambda}$  given in~\cite[Lemma 3]{BarJaiParVal20}. In fact, from~\cite[Lemma 3]{BarJaiParVal20}, there exists $\mu_{H_\Lambda} \in \Lambda_{H_\Lambda}$ and a positive constant $C$, independent of mesh parameters, such that
\[
\sup_{v\in H^1(\triH)} \frac{\langle \mu-\mu_{H_\Lambda}\,, v \rangle_{\btriH}}{\|v\|_{H^1(\triH)}} \leq C\,H_{\Lambda}^{k+1} | \A\nabla w |_{H^{k+1}(\triH)}.
\]
Then, using the definition of the norm $\vert\cdot \vert_{\Lambda}$, Poincar\'e inequality and estimate above, it holds that 
\begin{equation*}
\begin{aligned}
\vert\tmu-\tmu_{H_\Lambda}\vert_{\Lambda} &=  \sup_{\tphi\in \tH^{1/2}(\btriH)} \frac{\langle \tmu-\tmu_{H_\Lambda}, \tphi \rangle_{\btriH}}{ |\tphi |_{H^{1/2}(\btriH)}} = \sup_{\substack{\tv\in \tH^{1}(\triH) \\ \tv|_{\btriH} = \tphi}}\frac{\langle\tmu-\tmu_{H_\Lambda}, \tv\rangle_\btriH}{| \tv|_{H^{1}(\triH)}}  \\
&=\sup_{\tv\in \tH^{1}(\triH)}\frac{\langle\mu-\mu_{H_\Lambda},  \tv\rangle_\btriH}{| \tv|_{H^{1}(\triH)}} \leq C\,\sup_{\tv\in \tH^{1}(\triH)}\frac{\langle\mu-\mu_{H_\Lambda},  \tv\rangle_\btriH}{\| \tv\|_{H^{1}(\triH)}}  \\
&\leq C\,\sup_{v\in H^{1}(\triH)}\frac{\langle\mu-\mu_{H_\Lambda},  v\rangle_\btriH}{\| v\|_{H^{1}(\triH)}} \leq C\,H_{\Lambda}^{k+1} | \A\nabla w |_{H^{k+1}(\triH)}.
\end{aligned}
\end{equation*}
Also, we recall the approximation error associated with the Galerkin approximation at the second level
\begin{gather}
\vert(T-T_h)\tmu\vert_{H^{1}_{\A}(\triH)} \leq C\,h^{k+1} | T\tmu |_{H^{k+2}(\triH)}\text{ and }\vert(\widetilde T-\widetilde T_h)q \vert_{H^{1}_\A(\triH)} \leq C\,h^{k+1} | \widetilde T q |_{H^{k+2}(\triH)}\label{error-galer}
\end{gather}
where $q\in L^2(\OO)$ and $\tmu \in \tL$, under the assumption that $T\tmu$, $\widetilde Tq \in H^{k+2}(\triH)$. Therefore, using~\eqref{tlam} and setting $\tlambda\,|_\dtau=\A\nabla u\cdot\nn^\tau\,|_\dtau$ it holds that
\begin{equation}
\label{errorE}
\begin{aligned}
E(\tlambda) &= (1+\beta_{\max}^2)\,\inf_{\tmu_{H_{\Lambda}}\in\tL_{H_{\Lambda}}}
|\tlambda-\tmu_{H_{\Lambda}}|_\Lambda + \beta_{\max}^2|(T-T_h)\tlambda|_{H^{1/2}(\btriH)} \\
&\leq C\,\Big( H_{\Lambda}^{k+1}| \A\nabla u |_{H^{k+1}(\triH)}  + h^{k+1} | T\tlambda |_{H^{k+2}(\triH)} \Big) \\
&\leq C\,\Big[H_{\Lambda}^{k+1}|\A\nabla u|_{H^{k+1}(\triH)}
  +h^{k+1} \Big(| u |_{H^{k+2}(\triH)}+| \widetilde Tf |_{H^{k+2}(\triH)}\Big)\Big],
\end{aligned}
\end{equation}
where $E(\cdot)$ was defined in~\eqref{e:Edef}. 

Owing to the  interpolation estimates~\eqref{interp-bound} with $u$ replacing $w$ and~\eqref{error-galer} with $\tmu=\tlambda$ and $q=f$, we estimate   MH$^2$M  error using the $\PP_{k+1}(F)\times \PP_k(F) \times \mathbb{P}_{k+1} (\tau)$ element as follows.

\begin{theorem}
\label{conver}
Let  $(\rho,\lambda,u)$ be the exact solution of~\eqref{e:weak-hybridH} and $(\rho_{H_{\Gamma}},\lambda_{H_{\Lambda}},u_h)$ be  the MH$^2$M solution where $\rho_{H_{\Gamma}}\in \Lambda_{H_{\Lambda}}$ solves~\eqref{e:mhhmh}, and  $(u_h,\lambda_{H_{\Lambda}})\in V_h\times\Lambda_{H_{\Lambda}}$ is given in~\eqref{disc-sol}. Under the assumptions of Theorem~\ref{t:errorEstimates}, and assuming $u\in H^{k+2}(\triH)$, $\A\nabla u \in H^{k+1}(\triH)\cap \Hdiv$ and $\widetilde Tf \in H^{k+2}(\triH)$, with $k\geq 0$, it holds that 
\begin{gather*}
  \vert\rho-\rho_{H_{\Gamma}}\vert_{H^{1/2}(\btriH)}
  +\vert\lambda-\lambda_{H_{\Lambda}}\vert_{\Lambda}
  +\vert u-u_h\vert_{H^1_{\mathcal{A}}(\mathcal{T}_H)} 
\leq C \Big[ \Big(  H_\Gamma^{k+1}\,+ \, h^{k+1} \Big)\vert u \vert_{H^{k+2}(\triH)}
\\
		+\, H_{\Lambda}^{k+1}|\A\nabla u|_{H^{k+1}(\triH)} \,+ \,h^{k+1} | \widetilde Tf |_{H^{k+2}(\triH)}\Big].
\end{gather*}
\end{theorem}

\begin{proof}
The result is a direct consequence of Theorem~\ref{t:errorEstimates} and estimates~\eqref{interp-bound}-\eqref{errorE}.
\end{proof}
\begin{remark}[The $\PP_{k+1}(F)\times \PP_{k+1}(F) \times \mathbb{P}_{k+1} (\tau)$ case]
\label{rem:kkk}
Under the assumptions of Theorem~\ref{conver} with the additional regularity $\A\nabla u\in H^{k+2}(\triH)$ and following the arguments in the proof of Theorem~\ref{conver}, the  MH$^2$M solution with the element   $\PP_{k+1}(F)\times \PP_{k+1}(F) \times \mathbb{P}_{k+1} (\tau)$ satisfies
\begin{multline*}
\vert\rho-\rho_{H_{\Gamma}}\vert_{H^{1/2}(\btriH)}+\vert\lambda-\lambda_{H_{\Lambda}}\vert_{\Lambda}+ \vert u-u_h\vert_{H^1_{\mathcal{A}}(\mathcal{T}_H)}
\\
\leq C\Big[\Big(H_{\Gamma}^{k+1}+h^{k+1}\Big)|u|_{H^{k+2}(\triH)}+H_{\Lambda}^{k+2}|\A\nabla u|_{H^{k+2}(\triH)}+h^{k+1} | \widetilde Tf |_{H^{k+2}(\triH)}\Big].
\end{multline*}
Note that, since $H_{\Lambda} \leq H_{\Gamma}$ the leading error is $O(H_{\Gamma}^{k+1}+ h^{k+1})$ showing the element $\PP_{k+1} ( F)\times \PP_{k+1}(F) \times \mathbb{P}_{k+1} (\tau)$ provides no gain over the cheaper element $\PP_{k+1}( F) \times \PP_{k}(F) \times \mathbb{P}_{k+1} (\tau)$.
\end{remark}
\begin{remark}[The one-level MH$^2$M case]
Under the assumptions of Theorem~\ref{conver}, that $\widetilde V_h=\widetilde H^{1}(\mathcal{T}_H)$, $T_h=T$ and $\widetilde T_h=\widetilde T$, the error estimate in Theorem~\ref{conver} becomes
\begin{equation*}	\vert\rho-\rho_{H_{\Gamma}}^{one}\vert_{H^{1/2}(\btriH)}+\vert\lambda-\lambda_{H_{\Lambda}}^{one}\vert_{\Lambda}+ \vert u-u_h^{one}\vert_{H^1_{\mathcal{A}}(\mathcal{T}_H)} 
\leq C\Big[H_{\Gamma}^{k+1}\vert u \vert_{H^{k+2}(\triH)} + H_{\Lambda}^{k+1}|\A\nabla u|_{H^{k+1}(\triH)}  \Big]
\end{equation*}
when we use the M$H^2$M with the element \color{black} $\PP_{k+1}(F)\times\PP_{k}(F)$, \color{black} and 
\begin{equation*}	\vert\rho-\rho_{H_\Gamma}^{one}\vert_{H^{1/2}(\btriH)}+\vert\lambda-\lambda_{H_{\Lambda}}^{one}\vert_{\Lambda}
  + \vert u-u_h^{one}\vert_{H^1_{\mathcal{A}}(\mathcal{T}_H)} 
\leq C \Big(H_\Gamma^{k+1}\vert u \vert_{H^{k+2}(\triH)} + H_{\Lambda}^{k+2}|\A\nabla u|_{H^{k+2}(\triH)} \Big),
\end{equation*}
with the element \color{black} $\PP_{k+1}(F)\times \PP_{k+1}(F)$. \color{black}
\end{remark}



\section{Multiscale basis, degrees of freedom and algorithm}
\label{sec:basis}

Let  be $\{\rho_1,\dots,\rho_N\}$ be a basis for $\Gamma_{H_{\Gamma}}$ and  $\{\tlambda_1,\dots,\tlambda_{N^\prime}\}$ for $\tL_{H_{\Lambda}}$, where $N,\,N^\prime \in \NN^+$, and $\{\widetilde\psi_1,\dots,\widetilde\psi_{M}\}$ a base for $\widetilde V_h(\tau)$, with $\tau\in\triH$ and $M\in\NN^+$.  Let $\widetilde{\Gamma}_{H_{\Gamma}}:=\Gamma_{H_{\Gamma}}\cap\tH^{1/2}(\btriH)$ be the space of functions in $\Gamma_{H_{\Gamma}}$ with zero average in each element boundary.

The MH$^2$M computational algorithm writes: 
\begin{enumerate}[i)]
\item Compute $\lambda^0$ from~\eqref{u0l0};
\item Compute $T_h \tlambda_k =: \eta_k^{ms}$, for $k=1,\dots,N^\prime$, and $\widetilde T_h f =: \eta_f^{ms}$ from \eqref{e:Thdef}, i.e.,
\[
\sum_{j=1}^{M} v_j^k \int_\tau\A\bgrad(\widetilde \psi_i)\cdot\bgrad\widetilde \psi_j\,d\xx
=\langle\tlambda_k,\widetilde \psi_i\rangle_{\dtau}\quad\text{for all } i=1,\dots,M
\]
and
\[
 \sum_{j=1}^{M} q_j \int_\tau\A\bgrad(\widetilde \psi_i)\cdot\bgrad\widetilde \psi_j\,d\xx=\int_{\tau}f\widetilde \psi_i\:d\xx\quad\text{for all } i=1,\dots,M.
\]
Set 
\[
\eta_k^{ms} := \sum_{j=1}^{M}v_j^k \widetilde \psi_j \quad\text{and}\quad\eta_f^{ms} := \sum_{j=1}^{M} q_j  \widetilde \psi_j
\]
\item Compute $G_h\rho_k=:  \tlambda_k^{ms}$, for $k=1,\dots,N$, and $G_h \eta_f^{ms}=:  \tlambda_f^{ms}$, using $\eta_j^{ms}$  from  \eqref{e:Ghdef}, i.e.,
\[
\sum_{j=1}^{N^\prime}c_j^k \langle\tlambda_i, \eta_j^{ms}\rangle_\dtau
=\langle \tlambda_i,\rho_k\rangle_\dtau\quad\text{for all }i=1,\dots,N^\prime
\]
and 
\[
\sum_{j=1}^{N^\prime}d_j \langle\tlambda_i, \eta_j^{ms}\rangle_\dtau
=\langle  \tlambda_i,\eta_f^{ms}\rangle_\dtau\quad\text{for all }i=1,\dots,N^\prime.
\]
Set 
\[
\tlambda_j^{ms} := \sum_{k=1}^{N^\prime}c_j^k \tlambda_k \quad\text{and}\quad \tlambda_f^{ms} := \sum_{k=1}^{N^\prime} d_k  \tlambda_k\quad\text{with }j=1,\dots,N
\]
\item Solve~\eqref{e:rhodef}, i.e., compute $\alpha=\left\{\alpha_j \right\}_{j=1}^N$
\begin{equation}\label{e:rhohdef}
\sum_{j=1}^N\alpha_j\langle\tlambda_j^{ms},\rho_i\rangle_{\btriH}  =-\langle\lambda^0,\rho_i\rangle_\btriH+\langle \tlambda_f^{ms},\rho_i\rangle_{\btriH}\quad\text{with }i=1,\dots,N
\end{equation}
\item Compute $\rho_{H_{\Gamma}}$, $\tlambda_{H_{\Lambda}}$ and $\widetilde u_h$  
 \[\rho_{H_{\Gamma}}:=\sum_{i=1}^N\alpha_i\rho_i,\quad \tlambda_{H_{\Lambda}}:=\sum_{i=1}^N\alpha_i \tlambda_i^{ms}-\tlambda_f^{ms},\quad \widetilde u_h := \sum_{i=1}^N\alpha_i T_h \tlambda_i^{ms} -T_h\tlambda_f^{ms} +\eta_f^{ms} \] 
\item Compute $u_h^0$ from~\eqref{u0l0}
\item Compute $u_h$  and $\lambda_{H_\Lambda}$
\[ u_h   := u_h^0 + \sum_{i=1}^N\alpha_i T_h \tlambda_i^{ms} -T_h\tlambda_f^{ms} +\eta_f^{ms}\quad\text{and}\quad \lambda_{H_\Lambda} := \lambda^0 + \sum_{i=1}^N\alpha_i \tlambda_i^{ms}-\tlambda_f^{ms}\] 
or equivalently, 
\[ u_h :=   u_h^0 + \sum_{i=1}^N\alpha_i \sum_{j=1}^{N^\prime}c^i_j \sum_{k=1}^{M}v_k^j \widetilde \psi_k  -  \sum_{j=1}^{N^\prime}d_j  \sum_{k=1}^{M} v_k^j  \widetilde \psi_k   + \sum_{k=1}^{M} q_k  \widetilde \psi_k \]
and
\[ \lambda_{H_\Lambda} := \lambda^0 +\sum_{i=1}^N\alpha_i\sum_{j=1}^{N^\prime}c_j^i \tlambda_j -  \sum_{j=1}^{N} d_j  \tlambda_j. \]
\end{enumerate}

\begin{remark}
Some  points are worth mentioning:
\begin{itemize}
\item Although the solutions $u_h$ and $\lambda_{H_\Lambda}$ belong to the space $V_h$ and $\Lambda_{H_\Lambda}$ initially, they belong to the much smaller subspaces of dimension $N:=\dim(\Gamma_{H_ \Lambda})$. In fact, by construction, $\widetilde u_h \in \span\left\{\eta_1^{ms},\dots,\eta_N^{ms},\eta_f^{ms}\right\} \subset\widetilde V_h$ and $\tlambda_ {H_\Lambda} \in \span\left\{\tlambda_1^{ms},\dots,\tlambda_N^{ms}, \tlambda_f^{ms}\right\}\subset\tL_{H_\Lambda} $ , respectively.   Such a basis accounts for the multiscale characteristics of the problem, and are numerically approximated at the local level through steps $(i)-(iii)$ as the approximate formula for the base functions are not generally available. The MH$^2$M may be seen then as a member of reduced basis methods~\cite{BoyBriLelMadNguPat10}.
\item Steps $(i)-(iii)$ can be done in an offline stage. Furthermore, they are ``embarrassingly" parallel in the sense that the problems are local and entirely independent of each other. The global formulation $(iv)$ is the only global formulation because its degrees of freedom are nodal and associated with the coarse mesh skeleton;
\item All linear systems in steps $(ii)$-$(iv)$ are symmetric and positive definite. A closer look shows that two types of local linear systems must be assembled to solve the problems in $(ii)$-$(iv)$. The associated matrices can then be factored once and for all in an offline stage and then applied to the various right sides. This is particularly attractive when it comes to troubleshooting multiple queries as found in time marching process or in varying the source data $f$ in stochastic problems.
\end{itemize}
\end{remark}

\section{Numerical Results}
\label{sec:conv}

The following numerical experiments aim to access analytical predictions and compare the numerical performance of the MH$^2$M, MHM, MsFEM and FEM methods for two-dimensional problems. In what follows, we consider~\eqref{e:elliptic} with homogeneous Dirichlet boundary condition and $\Omega:=]0,1[\times]0,1[$. In the first set of tests the exact solution is known, and several convergence results are presented. We consider next a more demanding problem, that is, when the coefficients are oscillatory. 

\subsection{Problem with explicit solution}
For $\mathcal{A}=I$ and $f(x, y) :=-2x(x-1)-2y(y-1)$, the exact solution of~\eqref{e:elliptic} is 
\begin{equation*}
u(x,y)=x(x-1)y(y-1).
\end{equation*}

We start by considering polynomial approximation spaces for the unknown triple  $(\rho,\lambda,u)=(u|_{\btriH},\mathcal{A}\bgrad u\cdot\nn^\tau|_{\btriH},u)$ belonging to spaces $ H_0^{1/2}(\btriH)\times \Lambda\times H^1(\triH) $ using elements $\PP_{k+1}(F)\times\PP_k(F)\times\PP_{k+1}(\tau)$ and $\PP_{k+1}(F)\times\PP_{k+1}(F)\times\PP_{k+1}(\tau)$, with $k\geq 0$,  as described in Sections~\ref{ssec:ppp} and~\ref{ssec:kkk}, respectively. The approximations for traces of $u$ are polynomials of degree $k+1$ over the macro-edges and are continuous on the macro-skeleton. The discrete fluxes are piecewise discontinuous polynomials of order $k$ or $k+1$ over macro-edges. The  approximations for $u$ that are piecewise continuous polynomials of order $k+1$ within each macro-element, but might be discontinuous globally.

\color{black} We consider first a mesh refinement $H_\Gamma=H_\Lambda=H:=\sqrt2/2^{-j}$, $j=1,\,2,\, 3, \,4, \,5$ for elements of the form $\PP_{k+1}(F)\times\PP_k(F)\times\PP_{k+1}(\tau)$, with $h=H$ for $k=0,\,2$ and $h=H/2$ for $k=1$. We present in Figure~\ref{f10} the convergent results in the  $H^1(\T_H)$ semi-norm, displaying a $\mathcal{O}(H^{k+1})$ rate of convergence. Note that the element choices do not fulfill condition $(M2)$ for all $k$ and condition $(M1)$ for $k=0,\,2$, but are covered by  Remark~\ref{rem:relaxm1}. Numerical tests using the $\PP_{k+1}(F)\times\PP_{k+1}(F)\times\PP_{k+1}(\tau)$ approximation showed no noticeable improvements in the results, as anticipated in Section~\ref{ssec:kkk}, and are not displayed. \color{black}
\begin{figure}[h]
\center
\includegraphics[scale=0.7]{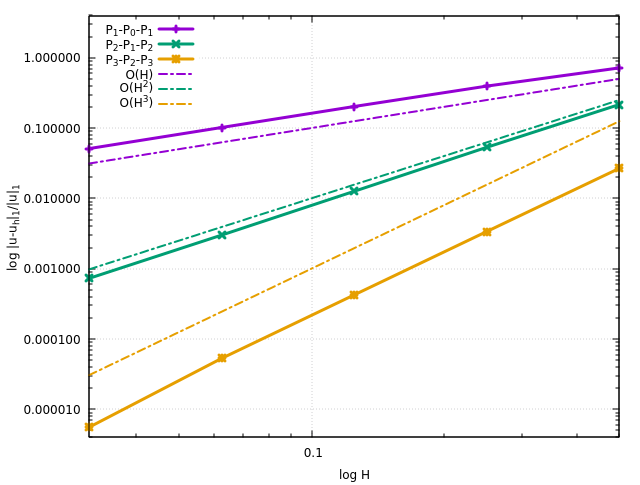}
\caption{Convergence results in the $H^1(\triH)$ relative semi-norm, with the $\PP_{k+1}(F)\times\PP_k(F)\times\PP_{k+1}(\tau)$ approximations. Here $k=0,\,1,\,2$. \color{black} The meshes are $H_\Gamma=H_\Lambda=H:=\sqrt 2/2^{-j}$, $j=1,\dots,5$, with $h=H$ for even $k$ and $h=H/2$ for $k=1$.  \color{black}}
\label{f10}
\end{figure}

The theory (see Theorem~\ref{conver} and Remark \ref{rem:kkk}) predicts that the coarse mesh can remain fixed while optimal convergence is achieved by refining the face meshes associated with the trace and flow variables, as well as the local submeshes, provided that Assumption $(M1)$ \color{black} and $(M2)$ are satisfied, for instance. \color{black} In this setting, we observe superconvergence with an additional convergence rate of $\mathcal{O}(H_\Gamma^{1/2})$, assuming the flux mesh and the local submeshes are ``sufficiently refined" relative to the trace mesh. Notably, unlike in the MHM method, the fluxes here are computed using local meshes, making them computationally affordable.

To illustrate this, consider the configuration $\PP_1(F)\times\PP_0(F)\times\PP_1(\tau)$, where the trace mesh is refined as $H_\Gamma=\sqrt{2}/2^{3+j}$ for $j=0,1,2,3$, while the flux and local meshes remain sufficiently refined. As shown in Figure~\ref{f40}, the results clearly exhibit the additional convergence rate of $\mathcal{O}(H_\Gamma^{1/2})$ in the (semi-)norms $H^1(\triH)$ and $L^2(\OO)$. This behavior of the MH$^2$M method extends beyond the scope of the current theoretical analysis. Nevertheless, it reflects a similar phenomenon observed in the MHM method, for which a superconvergence result has been rigorously proven~\cite{ChaParVal26}.
\begin{figure}[H]
\center
\includegraphics[scale=0.7]{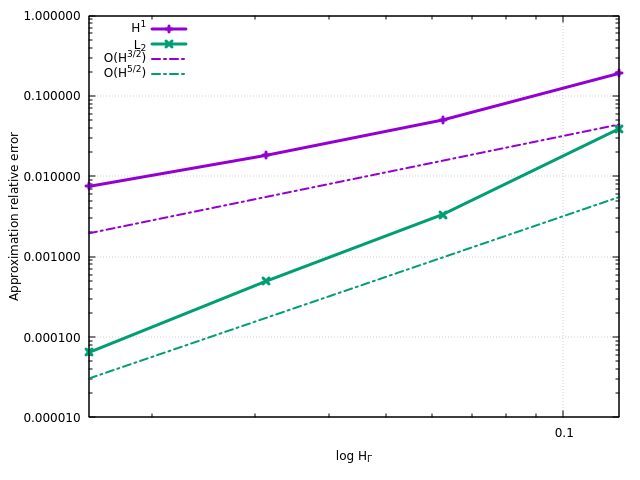}
\caption{$\PP_1(F)\times\PP_0(F)\times\PP_1(\tau)$ case: Approximation errors in relative $H^1(\triH)$ semi-norm and $L^2(\triH)$ norm, for  $\color{black} h,\,H_\Lambda\ll \color{black} H_\Gamma=\sqrt2/2^{3+j}$, $j=0,1,2,3$. We note an additional $\mathcal{O}(H_\Gamma^{1/2})$ convergence rates.}
\label{f40}
\end{figure}
%
%
%
Next, we compare the performance of the FEM/MsFEM, MHM, and MH$^2$M finite element methods, noting that for the Laplace equation, FEM and MsFEM yield identical results. The mesh size $H$ is kept fixed throughout, and we consider the approximation spaces \color{black} $\mathbb{P}_1(F)\times\mathbb{P}_0(F)\times\mathbb{P}_1(\tau)$ for the MH$^2$M and $\mathbb{P}_0(F)\times\mathbb{P}_1(\tau)$ for the MHM. \color{black}
The notations MHM-$i$ and MH$^2$M-$i$, with $i=1,2$, indicate the number of subdivisions per macro-edge, where $i=1$ corresponds to no subdivision.

In Figure~\ref{f50} (left), we display the relative $H^1(\triH)$ semi-norm error as a function of the size of the global matrix. We observe that the MH$^2$M method achieves a comparable error with significantly fewer degrees of freedom. In Figure~\ref{f50} (right), we compare the MHM and MH$^2$M methods using  \color{black} $\mathbb{P}_1(F)\times\mathbb{P}_0(F)\times\mathbb{P}_1(\tau)$ and $\mathbb{P}_0(F)\times\mathbb{P}_1(\tau)$ polynomial approximations.  \color{black} The relative $H^1(\triH)$ semi-norm error is plotted against the interface partition size $H_\Lambda=H_\Gamma=H/2^i$, for $i=0,1,2,3$. The numbers above each curve denote the size of the final system for the MH$^2$M method, while those below correspond to the MHM method. We note that both methods exhibit similar performance and display superconvergent behavior.

\begin{figure}[h]
\center
\includegraphics[scale=0.49]{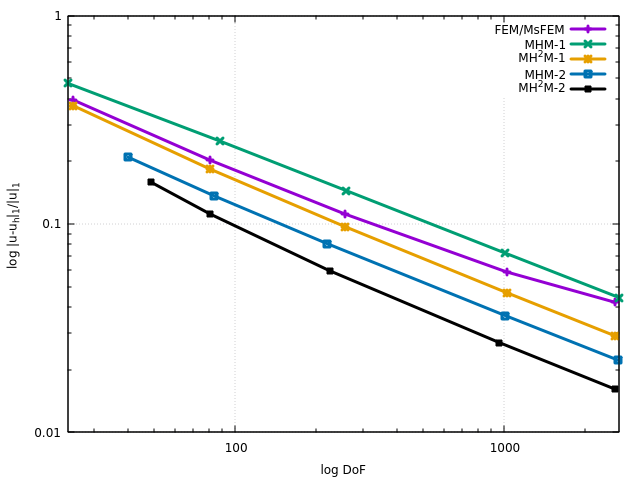}
\includegraphics[scale=0.54]{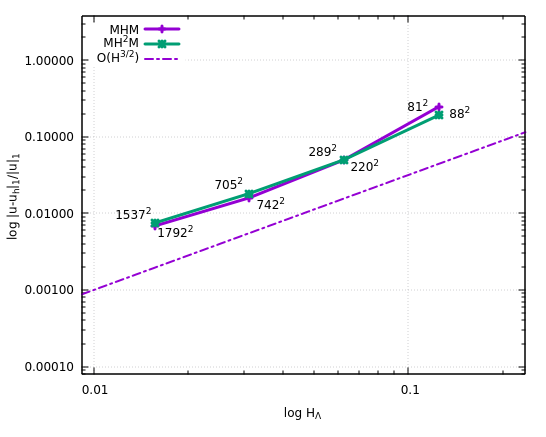}
\caption{Error curves with respect to the size of the final system (left) and to the interface partition (right). Here, MHM-$i$ and MH$^2$M-$i$, with $i = 1,\, 2$, indicate mesh partitions where $H_\Gamma = H_\Lambda = H/2^{i-1}$.}
\label{f50}
\end{figure}
%
%
\subsection{Oscillatory coefficient problem} 
\label{sec:oscill}
Consider now a non homogeneous coefficient of the form 
\begin{equation}\label{A(x/e)}
\A(x,y)=\frac{2+\gamma\sin(2\pi x/\varepsilon)}{2+\gamma \cos(2\pi y/\varepsilon)} + \frac{2+\sin(2\pi y/\varepsilon)}{2+\gamma\sin(2\pi x/\varepsilon)},
\end{equation}
where $\varepsilon>0$ is associated with the periodicity of $\A$. \color{black} In all cases, we set $\varepsilon = 1/14$. \color{black} We investigate the numerical performance of various multiscale methods, using a refined FEM solution as the reference with $\mathcal{H}=\sqrt{2}/128$. The surface plot of this reference solution is shown in Figure~\ref{OscSol} (left) and a surface of the MH$^2$M solution in Figure~\ref{OscSol} (right). \color{black} In this section, we used the elements $\mathbb{P}_1(F)\times\mathbb{P}_0(F)\times\mathbb{P}_1(\tau)$ for the MH$^2$M, and $\mathbb{P}_0(F)\times\mathbb{P}_1(\tau)$ for the MHM. \color{black}

\begin{figure}[H]
\center
\includegraphics[scale=0.3]{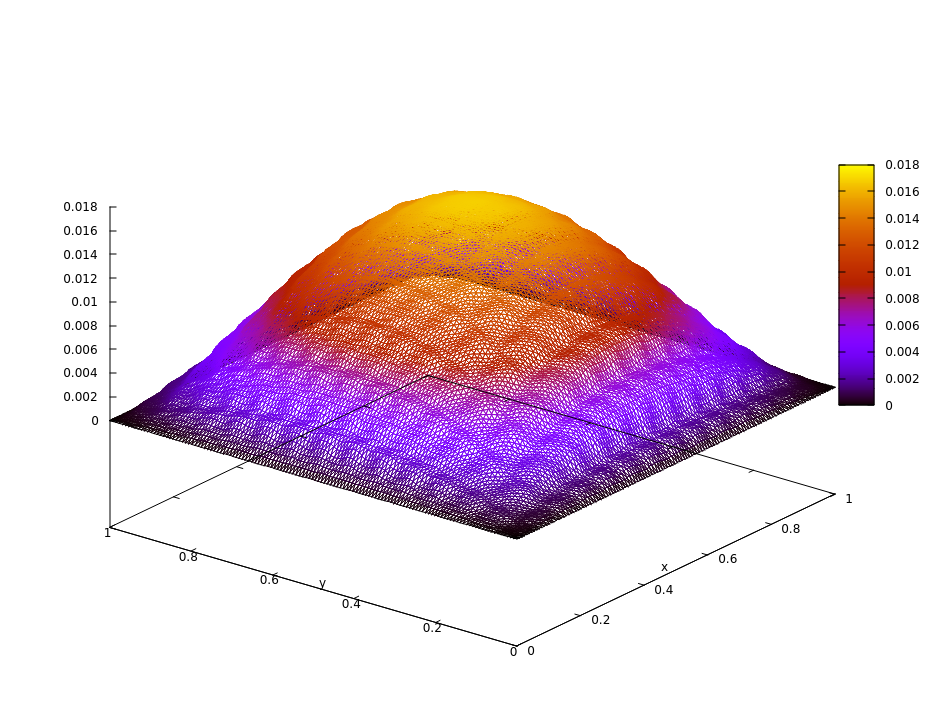}
\includegraphics[scale=0.39]{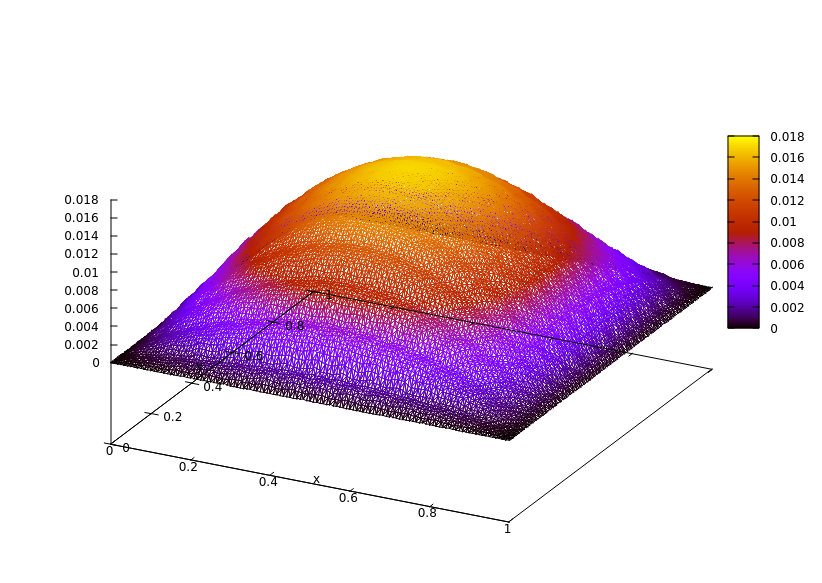}
\caption{Plot of the referential solution with $\mathcal{H}=\sqrt{2}/128$, on the left, and the numerical solution given by the MH$^2$M-$2$, where $H=\sqrt{2}/31$, $H_\Gamma=H_\Lambda=H/2$ and $h=H/4$.\label{OscSol}}
\end{figure}

We again consider the MHM-$i$ and MH$^2$M-$i$ schemes, where $i = 1,\, 4$ indicates mesh partitions defined by $H_\Gamma=H_\Lambda=H/2^{i-1}$. The results are shown in Figure~\ref{f:oscill}, and the label MH$^2$M$^\star$ refers to configurations where $H_\Lambda=H/8$. The value of $H$ varies across test cases to allow for a fair comparison based on the size of the final global system. The plots display cross-sections of the computed solutions at $y = 1/2$.  We observe that the MH$^2$M and MH$^2$M$^\star$ method provide more accurate results than the MHM method when using an equivalent number of degrees of freedom. 
\begin{figure}[h]
\center
\includegraphics[scale=0.46]{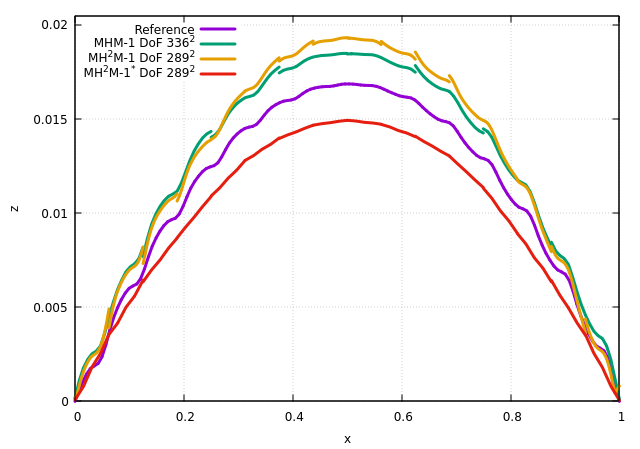}
\includegraphics[scale=0.46]{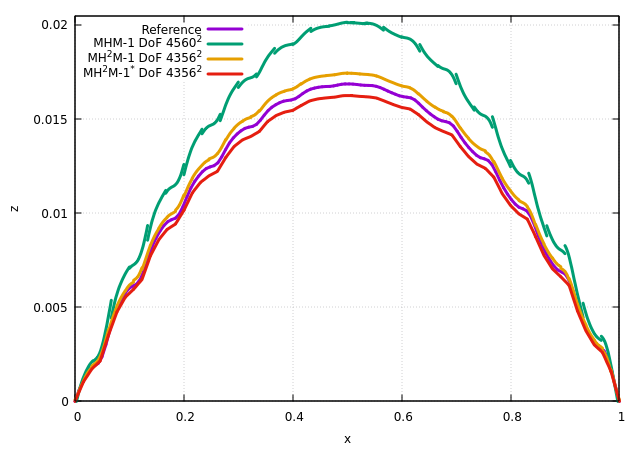}
\includegraphics[scale=0.46]{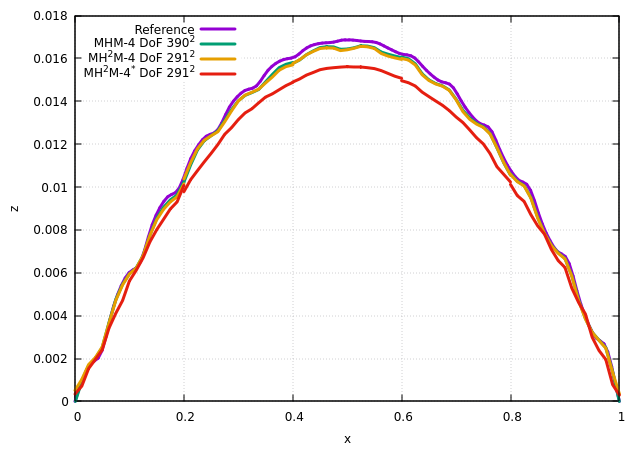}
\includegraphics[scale=0.46]{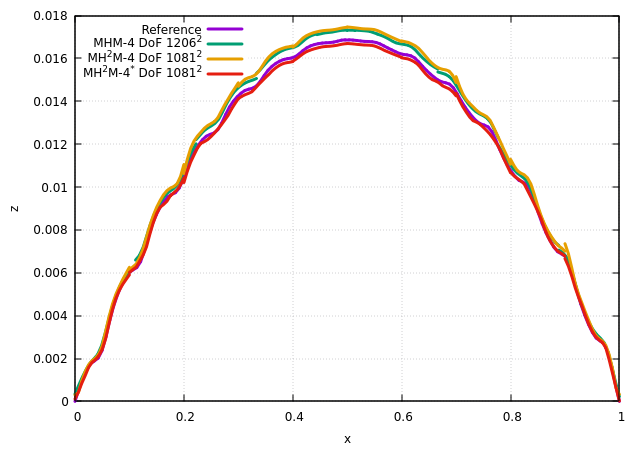}
\caption{Cross-section plots of the solutions at $y = 1/2$, with $\varepsilon=1/14$. On the top row, the fine-scale mesh size is $h = H/4$; on the bottom row, $h = H/8$. The MHM-i and MH$^2$M-i use $H_\Gamma=H_\Lambda=H/2^{i-1}$, and MH$^2$M-i$^*$ uses $H_\Gamma=H/2^{i-1}$ and $H_\Lambda=H/8$. In all cases, \color{black} DoF denotes the size of the global matrix.}
\label{f:oscill}
\end{figure}
 \color{black}
 
Figures~\ref{ress_mhm} and~\ref{ress_mh2m} display the numerical solutions obtained with the MHM-1 and MH$^2$M-1$^\star$ methods, respectively, as the size of the final global system increases. A noticeable resonance effect appears in the MHM results when $H\sim\sqrt{\varepsilon}$, in agreement with the findings of~\cite{ParValVer23}, whereas this phenomenon is not observed in the MH$^2$M-1$^\star$ results. This discrepancy may be attributed to the refinement of the face mesh for the flux, a strategy that was both observed and theoretically justified in~\cite{ParValVer23} for the MHM method. The key distinction is that, within the MH$^2$M framework, such refinement is performed locally and does not lead to an increase in the size of the global system. This contrast is noteworthy and motivates further theoretical investigation.
\begin{figure}[h]
\center
\includegraphics[scale=0.65]{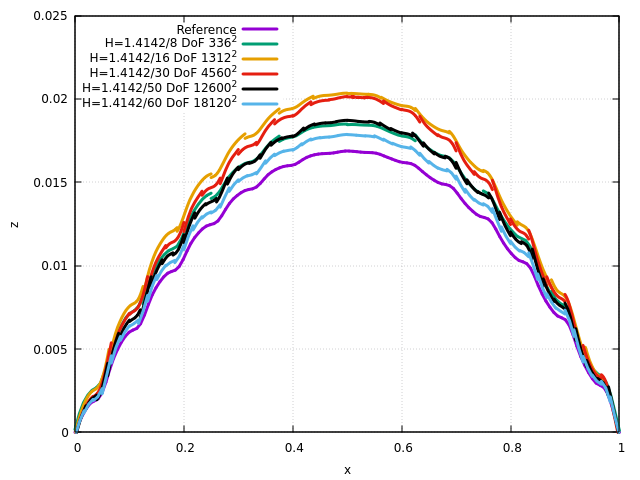}
\caption{Cross-section plots (at $y = 1/2$) of the solutions obtained with the MHM method for $\varepsilon = 1/14$. Resonance effects are clearly observed when $H\sim\sqrt{\varepsilon}\sim0.26$. } \label{ress_mhm}
\end{figure}
\begin{figure}[h]
\center
\includegraphics[scale=0.6]{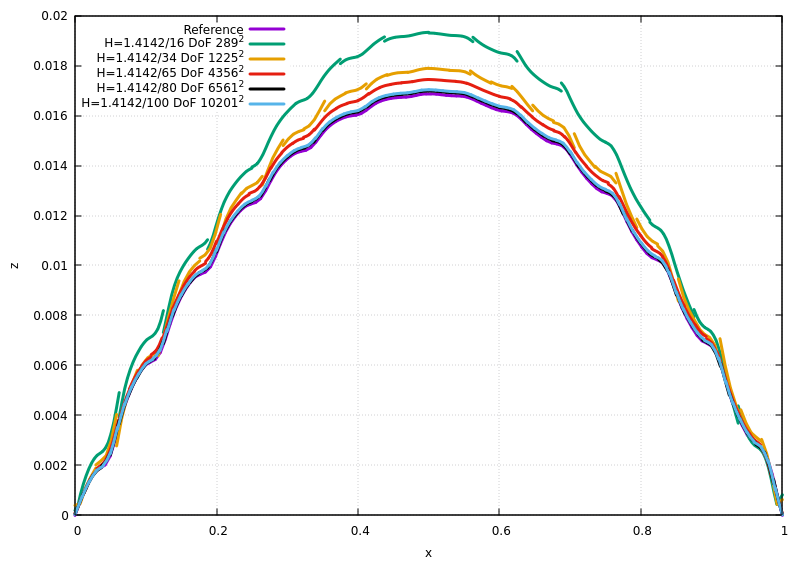}
\caption{Cross-section plots (at $y = 1/2$) of the solutions obtained with the MH$^2$M-1$^\star$ method for $\varepsilon = 1/14$. No resonance effects are observed.} \label{ress_mh2m}
\end{figure}

\section{Conclusion}
\label{sec:concl}

Brezzi and Marini's three-field method \cite{MR1262602} provided the background for new families of multiscale finite elements for the Poisson (Darcy) problem. The MH$^2$M incorporates a multiscale basis for the Lagrange multiplier associated with the flow variable, obtained from the multiscale basis of the original MHM (multiscale hybrid-mixed) method \cite{AraHarParVal13a}. Thanks to a new spatial decomposition, these multiscale bases are computed locally from positive definite algebraic systems and are ready to leverage parallel computational architectures. Unlike the MHM method, the global problem associated with MH$^2$M is coercive and the degrees of freedom are associated with the trace of the primal variable. Such similarity led us to establish a connection with MsFEM~\cite{EfeHouWu00}.

We proved that the MH$^2$M converges optimally with respect to the characteristic mesh lengths under reasonable conditions of compatibility between approximation spaces and under assumptions of local regularity (as usual in hybrid methods). Furthermore, we highlighted the interaction between the face mesh diameters ($H_{\Gamma}$ and $H_{\Lambda}$) and the second-level mesh parameter ($h$) on convergence. 

As verified by the numerical experiments, the MH$^2$M method requires fewer degrees of freedom than the MsFEM and MHM methods to achieve a given error level when using low-degree polynomial approximations. Superconvergence behavior was also observed when refining the face meshes associated with the trace variables, while keeping the coarse mesh fixed.  Furthermore, we have observed that MH$^2$M is less sensitive to resonance effects when approximating oscillatory solutions, in contrast to the MHM method. A rigorous analysis of the method's sensitivity to resonance phenomena lies beyond the scope of the present work and may be pursued by adopting the strategy proposed in~\cite{ParValVer17,ParValVer23}.

\section{Appendix}
\label{sec:appen}

\begin{lemma}\label{l:halfnorm}
Let $\tau\in\triH$ and $\xi\in H^{1/2}(\dtau)$. Then $|\xi|_{H^{1/2}(\dtau)}=|\phi_\xi|_{H_\A^1(\tau)}$, where $\phi_\xi$ weakly solves   
\begin{equation}\label{e:phixidef}
-\div\A\bgrad\phi_\xi=0\quad\text{in }\tau,
\qquad
\phi_\xi=\xi\quad\text{on }\dtau.  
\end{equation}
Moreover, if $\xi\in H_0^{1/2}(\btriH)$ and~\eqref{e:phixidef} holds for all $\tau$, then $\phi_\xi\in H_0^1(\OO)$ and
\[
|\xi|_{H^{1/2}(\btriH)}^2
=|\phi_\xi|_{H_\A^1(\OO)}^2
=\sum_{\tau\in\triH}|\xi|_{H^{1/2}(\dtau)}^2.
\]
\end{lemma}
\begin{proof}
Given $\xi\in H^{1/2}(\dtau)$, it is immediate that $|\phi_\xi|_{H_\A^1(\tau)}$ minimizes $|\cdot|_{H_\A^1(\tau)}$ over the space of $H^1(\tau)$ functions with trace $\xi$ on $\dtau$. Similarly, if $\xi\in H_0^1(\OO)$, then~\eqref{e:phixidef} holds in each element. But then $\phi_\xi\in H^1(\triH)$ and $\phi_\xi|_\btriH\in H_0^{1/2}(\btriH)$ implies that $\phi_\xi\in H_0^1(\OO)$. Then
\[
\sum_{\tau\in\triH}|\xi|_{H^{1/2}(\dtau)}^2
=\sum_{\tau\in\triH}|\phi_\xi|_{H_\A^1(\tau)}^2
=|\phi_\xi|_{H_\A^1(\OO)}^2. 
\]
\end{proof}
In the following lemma, we collect some useful technical results. 
\begin{lemma}\label{l:ids}
For a fixed $\tau\in\triH$ the following results hold:
\begin{enumerate}[(i)]
\item $|T\tmu|_{H^{1/2}(\dtau)}=|T\tmu|_{H_\A^1(\tau)}=|\tmu|_{H^{-1/2}(\dtau)}$ for all $\tmu\in\tL$; \label{e:id1}
\item $|G\xi|_{H^{-1/2}(\dtau)}=|\xi|_{H^{1/2}(\dtau)}$ for all $\xi\in H^{1/2}(\dtau)$; \label{e:id2}
\item $\langle\tmu,\xi\rangle_{\dtau}\le|\tmu|_{H^{-1/2}(\dtau)}|\xi|_{H^{1/2}(\dtau)}$ for all $\tmu\in\tL$ and $\xi\in H^{1/2}(\dtau)$. Also, for all $\xi\in H^{1/2}(\dtau)$, 
\begin{equation}
\label{e:id3}
\sup_{\tmu\in \tH^{-1/2}(\dtau)}\frac{\langle\tmu,\xi\rangle_{\dtau}}{|\tmu|_{H^{-1/2}(\dtau)}}
=|\xi|_{H^{1/2}(\dtau)};
\end{equation} 
\item $\langle\tmu,T\teta\rangle_\btriH\le|\tmu|_\Lambda|\teta|_\Lambda$ for all $\tmu$, $\teta\in\tL$.
\label{e:id4}
\end{enumerate}
\end{lemma}
\begin{proof}
The first identity in~\eqref{e:id1} follows from the definition of $T$ and Lemma~\ref{l:halfnorm}. To the second identity of~\eqref{e:id1} , we have:
\begin{multline*}
|\tmu|_{H^{-1/2}(\dtau)}
=\sup_{\psi\in\tH^{1/2}(\dtau)}\frac{\langle\tmu,\psi\rangle_\dtau}{|\psi|_{H^{1/2}(\dtau)}}
=\sup_{\substack{v\in\tH^{1}(\tau)}}\frac{\langle\tmu,v\rangle_\dtau}{|v|_{H^{1}_\A(\tau)}}
=\sup_{\substack{v\in\tH^{1}(\tau)}}\frac{\int_\tau\A\bgrad T\tmu\cdot\bgrad v\,d\xx}{|v|_{H^{1}_\A(\tau)}}
=|T\tmu|_{H_\A^1(\tau)}. 
\end{multline*}
Next,~\eqref{e:id2} follows from~\eqref{e:id1} with $\tmu=G\xi$ and~\eqref{e:tHident}. Item (iii) follows from the definition of the semi-norms $|\cdot|_{H^{1/2}(\dtau)}$ and  $|\cdot|_{H^{-1/2}(\dtau)}$. To show~\eqref{e:id3}, first denote by  $\phi_\xi$ the $\A$-harmonic extension of $\xi$, i.e., the solution of~\eqref{e:phixidef}. From~\eqref{e:id1}, we gather that   
\[
\frac{\langle\tmu,\xi\rangle_{\dtau}}{|\tmu|_{H^{-1/2}(\dtau)}}
=\frac{\int_\tau\A\bgrad T\tmu\cdot\bgrad\phi_\xi\,d\xx}{|T\tmu|_{H_\A^1(\tau)}}
\le|\phi_\xi|_{H_\A^1(\tau)}
=|\xi|_{H^{1/2}(\dtau)}
\]
where we used Lemma~\ref{l:halfnorm} at the last step. The identity with the supremum follows immediately. Finally, identity~\eqref{e:id4} follows from (iii) with $\xi =  T\teta$ and~\eqref{e:id1}, and the definition of the semi-norm $|\cdot |_{\Lambda}$ in \eqref{norm12d}.
\end{proof}

\color{black}
We provide below the details of the proof of the following Lemma. 
\begin{lemma}
\label{l0:m2}
Under Assumption  $(M2)$, there exists a mapping $\pi_\Lambda : \tL \rightarrow \tL_{H_{\Lambda_0}}$  such that for all $\tmu\in\tL$ and $\tau\in\triH$, it follows that
\begin{equation}
\label{piGB-demo}
\begin{gathered}
 \int_{\dtau} \pi_{\Lambda}(\tmu)\, \xi \,d\xx = \langle \tmu, \xi \rangle_{\dtau} \quad\text{for all }\xi \in \Gamma_{H_\Gamma}, \\
 |\pi_\Lambda(\tmu)|_{H^{-1/2}(\dtau)} \leq\alpha_\tau\, |\tmu |_{H^{-1/2}(\dtau)},
 \end{gathered}
\end{equation}
where $\alpha_\tau$ is  positive constant independent of mesh parameters.
\end{lemma}



\begin{proof}
Let $\tau$ be an element of $\triH$, $F$ an edge in $\EE_\Gamma(\dtau)$, and from Assumption $(M2)$, let $\mathfrak{f}_1$ and $\mathfrak{f}_2$ denote the two edges in $\EE_\Lambda(\dtau)$ such that
\[
F = \mathfrak{f}_1 \cup \mathfrak{f}_2 .
\]

Moreover, let $\mu_i^j$, $i=1,\ldots,k+1$, be a basis for the restriction of $\Lambda_{H_{\Lambda_0}}(\dtau)$ to the edge $\mathfrak{f}_j$, $j=1,2$, and let $\xi_l$, $l=1,\ldots,k+2$, be a local basis for $\Gamma_{H_{\Gamma}}(\dtau)|_F$ the restriction of $\Gamma_{H_{\Gamma}}(\dtau)$ to $F$.
Note that, given  $\xi\in\Gamma_{H_{\Gamma}}(\dtau)|_F$,  we get

\begin{equation}
\label{inject}
 \int_{\frak{f}_j}\mu_i^j\, \xi \,d\xx \,= \,0 \quad\text{for all } \mu_i^j \quad\Rightarrow  \xi = 0.
\end{equation}
To see this, first note that, since $\xi$ is $L^2(\mathfrak{f}_j)$-orthogonal to all polynomials of degree at most $k$, it can be represented on each $\mathfrak{f}_j$ by a Legendre polynomial of degree $k+1$. Therefore, if $\xi \neq 0$, then $\xi$ would have $2(k+1)$ distinct zeros on $F$ (namely, $k+1$ zeros on each edge $\mathfrak{f}_j$). This contradicts the fact that $\xi$ is a polynomial of degree $k+1$ on $F$, and~\eqref{inject} follows.

The rectangular moment matrix with entries
\[
\int_F \mu_i^j \, \xi_l \, d\xx,
\qquad i=1,\ldots,k+1, \quad l=1,\ldots,k+2, \quad j=1,\,2,
\]
represents the operator
\[
B:\Lambda_{H_{\Lambda_0}}(\dtau)|_F \rightarrow
\left(\Gamma_{H_\Gamma}(\dtau)|_F\right)'.
\]
By \eqref{inject}, its adjoint
\[
B^T:\Gamma_{H_\Gamma}(\dtau)|_F
\rightarrow
\left(\Lambda_{H_{\Lambda_0}}(\dtau)|_F\right)'
\]
is injective, and thus, $B$ is surjective. As a result, there exists a set of functions
$\{\mu_i\}_{i=1}^{k+2}$ spanning a subspace of
$\Lambda_{H_{\Lambda_0}}(\dtau)|_F$ such that the restriction of $B$ to this subspace is bijective.
Inspired by the dual-basis construction
of~\cite{Woh00}, we choose the local dual functions \(\mu_i\) such that 
%
\begin{equation}
\label{biortho}
\int_{F} \mu_i\, \xi_j \, d\xx \,=\, \delta_{ij}\qquad i,\,j =1,\ldots,k+2.
\end{equation}

Next, based on the local averaging construction underlying the 
Scott--Zhang interpolation operator~\cite{ScoZha90}, we define the 
following operators: Given \(\xi\in H^{1/2}(\dtau)\) and 
\(\mu\in H^{-1/2}(\dtau)\), consider the moments  
\[
\langle \mu, \xi_i \rangle_{\dtau}\quad\text{and}\quad \langle \mu_i, \xi \rangle_{\dtau},
\]
to  define the Fortin operator $\pi_\Lambda^\tau : H^{-1/2}(\dtau) \rightarrow \Lambda_{H_{\Lambda_0}}(\dtau)$ as follows
\begin{equation}
\label{piGam}
 \pi_\Lambda^\tau(\mu) := \sum_{i=1}^{N} \mu_i \langle \mu, \xi_i \rangle_{\dtau},
\end{equation}
and the mappings $\I_\Gamma: H^{1/2}(\dtau)\rightarrow \Gamma_{H_\Gamma}(\dtau)$ 
by
\begin{equation}
\label{ISZ}
\I_\Gamma(\xi) := \sum_{i=1}^N  \langle \mu_i, \xi \rangle_{\dtau}\xi_i
\end{equation}
for $\xi\in H^{1/2}(\dtau)$. 
Hereafter, $N:=\dim \Gamma_{H_\Gamma}(\dtau)$. Using the definition of $\pi_\Lambda^\tau(\cdot)$ in \eqref{piGam} and \eqref{biortho}, we get
\[
 \int_{\dtau} \pi_{\Lambda}^\tau(\mu)\, \xi_j \,d\xx = \sum_{i=1}^N  \int_{\dtau} \mu_i \xi_j  d\xx \, \langle \mu, \xi_i \rangle_{\dtau}= \sum_{i=1}^N \sum_{F\subset\dtau} \int_{F} \mu_i \xi_j  d\xx \, \langle \mu, \xi_i \rangle_{\dtau}=\langle \mu, \xi_j \rangle_{\dtau}.
\]

As a result, decomposing  $\xi \in \Gamma_{H_\Gamma}(\dtau)$ as
\[
\xi = \sum_{i=1}^N  c_i\xi_i\quad\text{where } c_i\in\RR,
\]
we get, for all $\mu \in H^{-1/2}(\dtau)$, 
\[
\int_{\dtau} \pi_{\Lambda}^\tau(\mu)\, \xi \,d\xx = \sum_{i=1}^N c_i \int_{\dtau} \pi_{\Lambda}^\tau(\mu)\, \xi_i \,d\xx = \langle \mu, \xi \rangle_{\dtau},
\]
which implies the first condition in \eqref{piGB-demo} replacing $\mu$ by $\tmu$ above. Note that the image of $\pi_\Lambda^\tau(\cdot)$ given in \eqref{piGam}  is a subset of  $\tL_{H_{\Lambda_0}}$ when it acts on space $\tL$.

Now, we turn   to prove the stability of $\pi_\Lambda^\tau(\cdot)$ in the $\tH^{-1/2}(\dtau)$ seminorm. First, observe that $\I_\Gamma(\cdot)$ is a projection operator  onto spaces $ \Gamma_{H_\Gamma}(\dtau)$  by \eqref{biortho}, and $\I_\Gamma(\cdot)$ corresponds to the adjoint of $\pi_\Lambda^\tau(\cdot)$. Indeed,
\begin{equation}
\label{eq1-demo}
\langle \mu, \I_\Gamma(\xi) \rangle_{\dtau} = \sum_{i=1}^N \langle \mu, \xi_i \rangle_{\dtau} \langle \mu_i, \xi \rangle_{\dtau} = \langle \pi_{\Lambda}(\mu), \xi \rangle_{\dtau}.
\end{equation}

Hence, assuming the existence of a positive constant $\alpha_\tau$, independent of $H_\tau$, such that 
\begin{equation}
\label{boundIG}
|\I_\Gamma(\xi)|_{H^{1/2}(\dtau)} \leq \alpha_\tau \, |\xi |_{H^{1/2}(\dtau)}\quad\forall \xi \in H^{1/2}(\dtau),
\end{equation}
then, we obtain  
\begin{equation*}
\begin{aligned}
|\pi_{\Lambda}^\tau(\tmu)|_{H^{-1/2}(\dtau)} & = \sup_{\txi\in \tH^{1/2}(\dtau)} \frac{\langle \pi_{\Lambda}^\tau(\tmu), \txi \rangle_{\dtau}}{ |\txi |_{H^{1/2}(\dtau)}} \\
& = \sup_{\txi\in \tH^{1/2}(\dtau)} \frac{\langle \tmu, \I_\Gamma(\txi) \rangle_{\dtau}}{ |\txi |_{H^{1/2}(\dtau)}} \\
& \leq  |\tmu|_{H^{-1/2}(\dtau)} \sup_{\txi\in \tH^{1/2}(\dtau)} \frac{| \I_\Gamma(\txi)|_{H^{1/2}(\dtau)}}{ |\txi |_{H^{1/2}(\dtau)}} \\
&\leq \alpha_\tau\, |\tmu|_{H^{-1/2}(\dtau)},
\end{aligned}
\end{equation*}
where we used the definition of the seminorm $|\cdot|_{H^{-1/2}(\dtau)}$, \eqref{eq1-demo} and \eqref{boundIG}. 
The result \eqref{piGB-demo} follows defining  $\pi_{\Lambda}(\cdot)$ such that $\pi_{\Lambda}(\cdot)|_\tau := \pi_{\Lambda}^\tau(\cdot)$ for all $\tau\in\triH$ restricted to space $\tL$.

We observe that \eqref{boundIG} is a consequence of the standard Scott--Zhang patchwise argument (see ~\cite{ScoZha90,Cia13}), together with interpolation theory (see~\cite{MCL00}[Chapter 3] and \cite{Heu14}). For completeness, we detail the proof in the present setting. For each edge $F\subset\dtau$, define
\[
 I(F) := \left\{i: \operatorname{supp}(\xi_i)\cap F\neq\emptyset\right\} \quad\text{and} \quad \omega_F := \cup_{i\in I(F)} F_i,
\]
and note that $\omega_F$ consists of at most the edge $F$ and its two neighboring edges. Moreover, the cardinality of $ I(F)$ is uniformly bounded  since the polynomial degree $k+1$ is fixed.

Next, for a fixed node $i$, select an edge $F_i$  such that $\operatorname{supp} \xi_i\cap F_i \neq\emptyset$, and let 
$\bar\mu_i$ be the associated local dual basis function, that is, the function satisfying \eqref{biortho}  with replaced $F$ by $F_i$.  Then, define $\mu_i\in\Lambda_{H_{\Lambda_0}}(\partial\tau)$ by setting
$\mu_i:=\bar\mu_i$ on $F_i$ and $\mu_i:=0$ on $\partial\tau\setminus F_i$.
Observe that $\{\xi_j|_{F_i}\}_{j\in I(F_i)}$ forms a basis for $\PP_{k+1}(F_i)$ and that
\[
\int_{\partial\tau}\mu_i\,\xi_j\,d\xx = \int_{F_i}\bar\mu_i\,\xi_j|_{F_i}\,d\xx = \delta_{ij}, \quad i,\,j=1,\ldots,N.
\]

From a standard scaling argument together with the  condition~\eqref{biortho}, and the fact that the lengths of the sub-edges $\mathfrak{f}_1$ and $\mathfrak{f}_2$ are uniformly comparable to that of $F_i$, we obtain
\begin{equation}
\label{scaling}
\|\xi_i\|_{L^2(F)} \leq C\, H_{F}^{1/2}\,
\quad\text{and}\quad \|\mu_i\|_{L^2(F_i)} \leq C\, H_{F_i}^{-1/2},
\end{equation}
where $H_F$ and $H_{F_i}$ are the diameters of $F$ and $F_i$, respectively. Assumption~$(M2)$ ensures that these diameters are uniformly comparable, and constants $C$ are independent of the mesh parameters. Using~\eqref{ISZ}, it follows that, for every edge $F\subset\partial\tau$,
\[
\I_\Gamma(\xi)|_F = \sum_{i\in I(F)} \langle \mu_i,\xi\rangle_{\partial\tau}\,\xi_i|_F = \sum_{i\in I(F)} \langle \mu_i,\xi\rangle_{F_i}\,\xi_i|_F,
\]
since $\operatorname{supp}\mu_i\subset F_i$. Then, applying the Cauchy--Schwarz inequality together with \eqref{scaling} and the uniform equivalence between the diameters $H_F$ and $H_{F_i}$, we obtain
\begin{equation*}
\begin{aligned}
\|\I_\Gamma(\xi)\|_{{L^2}(F)}&\leq \sum_{i\in I(F)} 
 \|\xi_i \|_{L^2(F)}  |\langle\mu_i,\xi\rangle_{F_i}|  \\
 &\leq \sum_{i\in I(F)} 
 \|\xi_i \|_{L^2(F)}  \|\mu_i\|_{L^2(F_i)} \|\xi\|_{L^2(F_i)}  \\
 & \leq C\,\|\xi \|_{L^2(\omega_F)}.
 \end{aligned}
\end{equation*}
 Squaring the previous inequality and summing over the edges gives

\begin{equation*}
\begin{aligned}
\| \I_\Gamma(\xi)\|_{{L^2}(\dtau)}^2 = \sum_{F\subset\dtau}\| \I_\Gamma(\xi)\|_{{L^2}(F)}^2
 & \leq C\,\sum_{F\subset\dtau} \|\xi \|_{L^2(\omega_F)}^2.
 \end{aligned}
\end{equation*}
Since the patches $\omega_F$ have uniformly bounded overlap, there exists a constant $C$, independent of $H_\tau$  and $N$,  such that
\begin{equation}\label{e:IboundL2}
\|\I_\Gamma(\xi)\|_{L^2(\dtau)}  \leq C\, \|\xi \|_{L^2(\dtau)}.
\end{equation}
In addition, since $\I_\Gamma(\cdot)$ is a projection, we have
$ \I_\Gamma(\xi_F)|_F=\xi_F$, where $\xi_F:=\frac{1}{|\omega_F|}\int_{\omega_F}\xi\,d\xx.$ Consequently,
$\I_\Gamma(\xi)|_F = \xi_F+\I_\Gamma(\xi-\xi_F)|_F$. Then, for every edge $F\subset\omega_F$, it follows that
\[
\begin{aligned}
|\I_\Gamma(\xi)|_{H^1(F)} &= |\I_\Gamma(\xi-\xi_F)|_{H^1(F)} \\
&\le \frac{C}{H_F}\,\|\I_\Gamma(\xi-\xi_F)\|_{L^2(F)} \\
&\le \frac{C}{H_F}\,\|\xi-\xi_F\|_{L^2(\omega_F)} \\
&\le C\,|\xi|_{H^1(\omega_F)},
\end{aligned}
\]
where we used an inverse inequality, the $L^2$-stability of $\I_\Gamma(\cdot)$ established above, and the Poincar\'e inequality on each edge contained in $\omega_F$. Squaring the above inequality and summing over all edges yields
\begin{equation}\label{e:IboundH1}
|\I_\Gamma(\xi)|_{H^1(\partial\tau)} \le C\,|\xi|_{H^1(\partial\tau)}.
\end{equation}\
As a result, the operator $\I_\Gamma(\cdot)$ is stable in both
$L^2(\partial\tau)$ and $H^1(\partial\tau)$, with constants independent of the mesh parameters. 
By the interpolation identity
\[
H^{1/2}(\partial\tau)
=
[L^2(\partial\tau),H^1(\partial\tau)]_{1/2},
\]
and the equivalence between the interpolation norm and the trace norm on $H^{1/2}(\partial\tau)$, estimate~\eqref{boundIG} follows from~\eqref{e:IboundL2} and~\eqref{e:IboundH1}
\end{proof}
\color{black}

 \color{black}

\bibliographystyle{plain}
\bibliography{AFF.bib}

\end{document}